\documentclass[12pt]{article}
\usepackage{amssymb,amsmath,amsthm,tikz,multirow, mathdots}
\usepackage{enumerate} 
\usepackage{hyperref} 
\usepackage [autostyle, english = british]{csquotes}
\usepackage [english]{babel}
\usepackage{mathtools}
\usepackage{bm}
\usepackage{authblk} 

\usepackage{comment} 

\usepackage{tikz-3dplot}
\usetikzlibrary{calc,arrows, arrows.meta, math, backgrounds} 

\newtheorem*{theorem*}{Theorem}

\theoremstyle{definition}
\newtheorem*{definition*}{Definition}
\newtheorem*{case*}{Case}
\newtheorem*{subcase*}{Subcase}
\newtheorem*{subsubcase*}{Subsubcase}

\theoremstyle{plain}
\newtheorem{thm}{Theorem}[section]
\newtheorem{lem}[thm]{Lemma}
\newtheorem{prop}[thm]{Proposition}

\theoremstyle{definition}

\theoremstyle{remark}

\numberwithin{equation}{section}

\newcommand{\AVC}{\text{AVC}} 

\newcommand{\quotes}[1]{``#1''} 

\newcommand{\AxisRotator}[1][rotate=0]{
	\tikz [x=0.25 cm, y=0.75 cm, line width=0.2 ex, -stealth, #1] \draw (0,0) arc (-150:150:0.5 and 0.5);
}

\providecommand{\keywords}[1]{\noindent \textit{Keywords:} #1}
\providecommand{\subject}[1]{\noindent \textit{Mathematics Subject Classification:} #1}

\title{Dihedral Tilings of the Sphere by Regular Polygons and Quadrilaterals I: Squares and Rhombi}
\author[]{Hoi Ping LUK}
\affil[]{The Hong Kong University of Science \& Technology} 
\affil[]{email: hoi@connect.ust.hk}

\begin{document}
\maketitle

\begin{abstract} We classify the dihedral edge-to-edge tilings of the sphere by squares and rhombi. \\

 \\

\subject{05B45, 52C20, 51M09, 51M20}
\end{abstract}

\section{Introduction}

There is a long fascination on tilings throughout human history. Recently, we see three major breakthroughs in the studies of tilings. One of them is the discovery of aperiodic monotile of the plane \cite{smkgs} and the other two are classifications of tilings of the sphere. In the classifications, one is the complete classification of spherical tilings by regular polygons \cite{aehj, joh, zal}. The other is the complete classification of the monohedral edge-to-edge spherical tilings, which was pioneered by Sommerville \cite{som} a century ago and has been settled by a collective effort \cite{awy, ay, cl, cl2, cly, gsy, ua, wy, wy2}. 

This paper is the first of the series to classify dihedral tilings of the sphere, where one prototile is a regular polygon and the other is a rhombus. The two prototiles in \cite{luk} are one regular polygon and a quadrilateral with two distinct edge lengths. In this paper, the two prototiles are one square and one rhombus. Both are equilateral quadrilaterals with edge combination $x^4$. The square (shaded) and the rhombus (unshaded) are illustrated in Figure \ref{Fig-quad-a4-angles}. Throughout our discussion, the shaded tiles are always squares. We assume that the degree of a vertex is $\ge3$.

\begin{figure}[h!] 
\centering
\begin{tikzpicture}

\tikzmath{
\r=0.8;
\gon=4;
\th=360/\gon;
\x=\r*cos(\th/2);
}

\begin{scope} 

\fill[gray!40]
	(0.5*\th:\r) -- (1.5*\th:\r) -- (2.5*\th:\r)  -- (3.5*\th:\r) -- cycle
;

\foreach \a in {0,...,3} {
\draw[rotate=\th*\a]
	(90-0.5*\th:\r) -- (90+0.5*\th:\r) 
;

}

\foreach \c in {0,...,3} {

\node at (0.5*\th+\th*\c:0.65*\r) {\small $\alpha$};

\node at (\th+\th*\c: 1*\r) {\small $x$}; 

}

\end{scope}

\begin{scope}[xshift=3cm] 

\foreach \a in {0,...,3} {
\draw[rotate=\th*\a]
	(90-0.5*\th:\r) -- (90+0.5*\th:\r) 
;

\node at (\th+\th*\a: 1*\r) {\small $x$}; 

}

\foreach \c in {0,1} {

\node at (0.5*\th+2*\th*\c:0.65*\r) {\small $\gamma$};

}

\node at (1.55*\th:0.625*\r) {\small $\beta$};
\node at (3.5*\th:0.625*\r) {\small $\beta$};

\end{scope}

\end{tikzpicture}
\caption{The prototiles: square and rhombus}
\label{Fig-quad-a4-angles}
\end{figure}
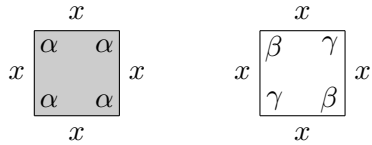


For simplicity, by {\em tilings} we mean edge-to-edge dihedral tilings of the sphere unless otherwise specified.

The main result is given below, where $f$ denotes the number of tiles. 

\begin{theorem*} The edge-to-edge dihedral tilings of the sphere by squares and rhombi are 
\begin{enumerate}[I.]
\item Earth map type: one infinite family of tilings with $f=8c-2$, where $c\ge2$, and vertices $\{ \beta^2\gamma, \alpha\beta\gamma^c  \}$;
\item Platonic and Archimedean type:
\begin{itemize}
\item one cube with $f=6$, and vertex $\{ \alpha\beta\gamma \}$;
\item two triangular fusions of the snub cube with $f=22$, and vertices $\{ \alpha\beta^2, \alpha\beta\gamma^2 \}$;
\item one quadrilateral subdivision of the truncated octahedron with $f= 30$, and vertices $\{ \beta^3, \alpha\beta\gamma^2 \}$;
\end{itemize}
\item Sporadic type: two tilings with $f=14$, and vertices $\{ \alpha^2\beta, \alpha^3\gamma \}$.
\end{enumerate}
\end{theorem*}


\begin{figure}[h!]
\centering
\begin{tikzpicture}[scale=1]

\tikzmath{
\s=1;
}

\begin{scope}[yshift=-1*\s cm] 

\tikzmath{ 
\r=0.8;
\x = sqrt(2*\r^2);
\y = sqrt(\r^2 - (\x/2)^2);
\yy = 2*\r + \y;
\l=2*\r*sin(67.5);
\L=sqrt(\r^2+\l^2-2*\r*\l*cos(135+22.5));
\A=acos((\L^2+\l^2-\r^2)/(2*\L*\l));
\tz=2;
\tzz=\tz-1;
\tzzz=\tzz-1;
}

\draw[]
	(90:\r) -- (90+22.5:\l)
	(90+22.5:\l) -- (90+22.5-\A:\L)
	(1*\x,0) -- (1*\x, \r)
;

\foreach \a in {0,...,\tz} {

\draw[xshift=\x*\a cm]
	(0:0) -- (90:\r)
	(90:\r) -- (90-22.5:\l)
	(90-22.5:\l) -- (90-22.5+\A:\L)
;
	
}

\foreach \a in {0,...,\tzz} {

\draw[xshift=\x*\a cm]
	(90-22.5:\l) -- (\x,\r)
;

}

\foreach \c in {-1,0,...,\tzz} {

\node at (1*\x+\c*\x, 2*\r + 0.75*\y) {\small $\gamma$}; 
\node at (1*\x+\c*\x, 1*\r + 0.5*\y) {\small $\gamma$}; 
\node at (0.65*\x+\c*\x, 1.1*\r + 1*\y) {\small $\beta$}; 
\node at (1.35*\x+\c*\x, 1.1*\r + 1*\y) {\small $\beta$}; 

}

\foreach \c in {0,...,\tzz} {

\node at (0.5*\x+\c*\x, 0.1*\r) {\small $\gamma$}; 
\node at (0.5*\x+\c*\x, 1*\r + 0.5*\y) {\small $\gamma$}; 
\node at (0.15*\x+\c*\x, 0.9*\r) {\small $\beta$}; 
\node at (0.85*\x+\c*\x, 0.9*\r) {\small $\beta$}; 

}

\draw[<-, gray]
	(-\y, \L+0.4*\r) -- (\tz*\y-0.25*\y, \L+0.4*\r)
;

\draw[->, gray]
	(\tz*\y+0.25*\y,1*\L+0.4*\r) -- (2*\tz*\y+\y,1*\L+0.4*\r)
;

\node at (\tz*\y, \L+0.4*\r) {$\ast$};

\draw[<-, gray]
	(0,-0.4*\r) -- (\tz*\y-0.25*\y,-0.4*\r)
;

\draw[->, gray]
	(\tz*\y+0.25*\y,-0.4*\r) -- (2*\tz*\y,-0.4*\r)
;

\node at (\tz*\y, -0.4*\r) {$\star$};

\node at (\tz*\y, -1.25*\r) {$\mathcal{T}$};

\end{scope} 

\begin{scope}[xshift=6.5*\s cm] 

\tikzmath{
\L=0.55;
\th=360/4;
\X=\L*cos(0.5*\th);
\l=2*\X;
}

\filldraw[gray!50]
	(0,0) circle (2.5*\l)
;

\filldraw[white]
	(0,0) circle (2.24*\l)
;

\fill[gray!50]
	(0.5*\l,0.5*\l) -- (0.5*\l,-0.5*\l) -- (-0.5*\l,-0.5*\l) -- (-0.5*\l,0.5*\l) -- cycle
;

\foreach \a in {0,...,3}
{
\begin{scope}[rotate=90*\a]

\draw
	(0.5*\l,-0.5*\l) -- (0.5*\l,0.5*\l) -- (1*\l,1*\l) -- (1*\l,2*\l) -- ++(-0.4*\l,-0.3*\l)
	(1*\l,1*\l) -- ++(0.3*\l,-0.3*\l);

\node at (0.4*\l,0.68*\l) {\small $\ast$};

\node at (0,1.4*\l) {\small $\mathcal{T}$};

\node[] at (55:2.05*\l) {\small $\star$};

\end{scope}

}	

\draw (0,0) circle (2.24*\l);

\end{scope}

\end{tikzpicture}
\caption{The earth map type tilings: $\ast = \gamma^c$ and $\star = \gamma^{c-1}$ for $c\ge2$}
\label{Fig-a4-families-earth-map-tilings}
\end{figure}
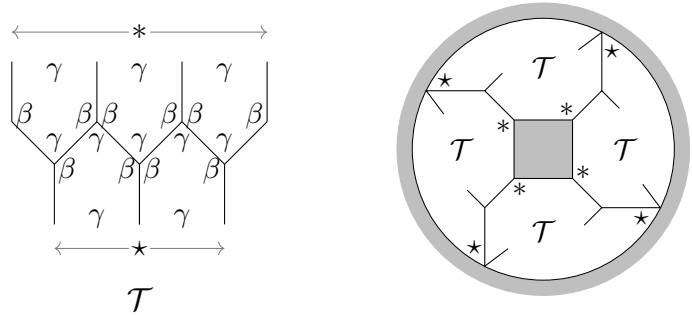

The family of earth map type are derived from the monohedral tilings $E_{\square}4$ in \cite{cly}. The first picture of Figure \ref{Fig-a4-families-earth-map-tilings} shows a part of $E_{\square}4$ with $\gamma^3$ at one end and $\gamma^2$ at the other. Let $\mathcal{T}$ denote the general version of it consisting of $2c-1$ rhombi with $\gamma^c$ at one end and $\gamma^{c-1}$ at the other (the first picture shows $c=3$). The earth map type tilings are constructed by four copies of $\mathcal{T}$ and $2$ squares. The four $\mathcal{T}$'s are glued together between the two squares as illustrated in the second picture.

The Platonic type tiling is the first picture of Figure \ref{Fig-a4-platonic-archimedean-tilings}. The three Archimedean type tilings are the second, third and fourth picture. The two sporadic tilings are illustrated in Figure \ref{Fig-a4-sporadic-tilings}. Among them, $\beta$ is marked by \quotes{$\diamond$}. 


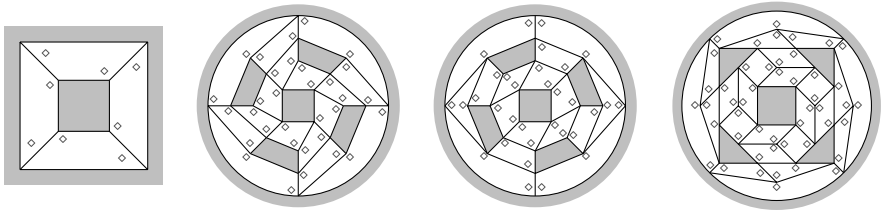
\begin{figure}[h!]
\centering
\begin{tikzpicture}[scale=0.6]

\tikzmath{
\s=1;
\ss=0.8;
\sss=0.75;
}

\begin{scope}[]

\begin{scope}[scale=\ss] 

\tikzmath{
\s=1;
\r=1;
\th=360/4;
\sh=1;
\ss=2.2;
}

\fill[gray!50, scale=1]
	(\ss*\sh,\ss*\sh) -- (\ss*\sh,-\ss*\sh) -- (-\ss*\sh,-\ss*\sh) -- (-\ss*\sh,\ss*\sh) -- cycle
;

\fill[white]
	(0.5*\th:2.5*\r) -- (1.5*\th:2.5*\r) -- (2.5*\th:2.5*\r)  -- (3.5*\th:2.5*\r) -- cycle
;

\fill[gray!50]
	(0.5*\th:\r) -- (1.5*\th:\r) -- (2.5*\th:\r)  -- (3.5*\th:\r) -- cycle
;

\foreach \a in {0,...,3} {

\draw[rotate=\th*\a]
	(0.5*\th:\r) -- (1.5*\th:\r)
	(0.5*\th:\r) -- (0.5*\th:2.5*\r)
	(0.5*\th:2.5*\r) -- (1.5*\th:2.5*\r)
;

\node at (0.5*\th+0.15*\th+\a*\th: 1.1*\r) {\scalebox{0.5}{$\diamond$}};
\node at (0.5*\th-0.1*\th+\a*\th: 1.8*\r) {\scalebox{0.5}{$\diamond$}};

}

\end{scope}

\end{scope}

\begin{scope}[xshift=4.75*\s cm] 

\tikzmath{
\s = 1;
\r= 0.5;
\th=360/4;
\ph=360/8;
\x=\r*cos(0.5*\th);
}

\fill[gray!50] (0,0) circle (4.5*\r);
	
\fill[white] (0,0) circle (4*\r);

\fill[gray!50]
	(0.5*\th:\r) -- (1.5*\th:\r) -- (2.5*\th:\r) -- (3.5*\th:\r) -- cycle
;

\foreach \a in {0,...,3} {

\fill[gray!50, rotate=\a*\th]
	(90:2*\r) -- (90:3*\r) -- (90-0.5*\th:3*\r) -- (90-0.5*\th:2*\r) -- cycle
;

}

\foreach \a in {0,...,3} {

\draw[rotate=\a*\th]
	(0.5*\th:\r) -- (1.5*\th:\r)
	(1.5*\th:\r) -- (90:2*\r)
	(0.5*\th:\r) -- (90-\ph:2*\r)
	(90:3*\r) -- (90:4*\r)
	(90+0.5*\th:3*\r) -- (90:4*\r)
;


}

\foreach \a in {1,3,...,7} {

\draw[rotate=\a*\ph]
	(90:2*\r) -- (90+\ph:2*\r)
	(90:3*\r) -- (90+\ph:3*\r)
;

}

\foreach \a in {0,2,...,6} {

\tikzset{rotate=\a*\ph}

\draw[]
	(90+0.5*\th:2*\r) -- (90:3*\r)
;

\draw[gray, densely dashed]
	%
;

\node at ($(90:2*\r) !1/4! (90-0.5*\th:\r)$) {\scalebox{0.5}{$\diamond$}}; 
\node at ($(90:2*\r) !4/5! (90-0.5*\th:\r)$) {\scalebox{0.5}{$\diamond$}}; 

\node at (90+0.1*\th:2*\r) {\scalebox{0.5}{$\diamond$}}; 
\node at (1.5*\th-0.1*\th:1.75*\r) {\scalebox{0.5}{$\diamond$}}; 

\node at (90+0.05*\th:3.1*\r) {\scalebox{0.5}{$\diamond$}}; 
\node at (1.5*\th-0.075*\th:2.725*\r) {\scalebox{0.5}{$\diamond$}}; 

\node at (0.5*\th:3.25*\r) {\scalebox{0.5}{$\diamond$}}; 
\node at (90-0.05*\th:3.75*\r) {\scalebox{0.5}{$\diamond$}}; 

}

\foreach \a in {0,...,7} {

\draw[rotate=\a*\ph]
	(90:2*\r) -- (90:3*\r)
;

}

\draw[] (0,0) circle (4*\r);

\end{scope}

\begin{scope}[xshift=10*\s cm]

\tikzmath{
\s = 1;
\r= 0.5;
\th=360/4;
\ph=360/8;
\x=\r*cos(0.5*\th);
}

\fill[gray!50] (0,0) circle (4.5*\r);
	
\fill[white] (0,0) circle (4*\r);

\fill[gray!50]
	(0.5*\th:\r) -- (1.5*\th:\r) -- (2.5*\th:\r) -- (3.5*\th:\r) -- cycle
;

\foreach \a in {0,...,3} {

\fill[gray!50, rotate=\a*\th]
	(90:2*\r) -- (90:3*\r) -- (90+0.5*\th:3*\r) -- (90+0.5*\th:2*\r) -- cycle
;

}

\foreach \a in {0,...,3} {

\draw[rotate=\a*\th]
	(0.5*\th:\r) -- (1.5*\th:\r)
	%
	%
;

}

\foreach \a in {0,...,7} {

\draw[rotate=\a*\ph]
	(90:2*\r) -- (90+\ph:2*\r)
	(90:2*\r) -- (90:3*\r)
	(90:3*\r) -- (90+\ph:3*\r)
;

}

\foreach \a in {0,1} {

\tikzset{rotate=\a*180}

\draw[]
	(0.5*\th:\r) -- (0.5*\th:2*\r)
	(1.5*\th:\r) -- (90:2*\r)
	(1.5*\th:\r) -- (180:2*\r)
	(90:3*\r) -- (90:4*\r)
	(0:4*\r) -- (0.5*\th:3*\r) 
	(0:4*\r) -- (-0.5*\th:3*\r) 
;

\draw[gray, densely dashed]
	%
	%
;



\node at ($(90:2*\r) !1/4! (90-0.5*\th:\r)$) {\scalebox{0.5}{$\diamond$}}; 
\node at ($(90:2*\r) !4/5! (90-0.5*\th:\r)$) {\scalebox{0.5}{$\diamond$}}; 

\node at (1.5*\th:1.25*\r) {\scalebox{0.5}{$\diamond$}}; 
\node at (1.5*\th:1.75*\r) {\scalebox{0.5}{$\diamond$}}; 

\node at (0.5*\th-0.15*\th:1.15*\r) {\scalebox{0.5}{$\diamond$}}; 
\node at (0.05*\th:1.65*\r) {\scalebox{0.5}{$\diamond$}}; 

\node at (0.5*\th+0.075*\th:2.15*\r) {\scalebox{0.5}{$\diamond$}}; 
\node at (90-0.075*\th:2.625*\r) {\scalebox{0.5}{$\diamond$}}; 

\node at (1.5*\th+0.075*\th:2.15*\r) {\scalebox{0.5}{$\diamond$}}; 
\node at (180-0.075*\th:2.65*\r) {\scalebox{0.5}{$\diamond$}}; 

\node at (0:3.2*\r) {\scalebox{0.5}{$\diamond$}}; 
\node at (0:3.7*\r) {\scalebox{0.5}{$\diamond$}}; 

\node at (90-0.05*\th:3.65*\r) {\scalebox{0.5}{$\diamond$}}; 
\node at (0.5*\th:3.2*\r) {\scalebox{0.5}{$\diamond$}}; 

\node at (90+0.05*\th:3.65*\r) {\scalebox{0.5}{$\diamond$}}; 
\node at (1.5*\th:3.2*\r) {\scalebox{0.5}{$\diamond$}}; 

}

\draw[] (0,0) circle (4*\r);

\end{scope} 

\begin{scope}[xshift=15.35*\s cm ]

\tikzmath{
\s = 1;
\r= 0.6;
\th=360/4;
\ph=360/8;
\x=\r*cos(0.5*\th);
}

\fill[gray!50] (0,0) circle (3.85*\r);

\fill[white] (0,0) circle (3.5*\r);

\fill[gray!50]
	(0.5*\th:\r) -- (1.5*\th:\r) -- (2.5*\th:\r) -- (3.5*\th:\r) -- cycle
;

\foreach \a in {0,1,2,3} {

\fill[gray!50, rotate=\a*\th]
	(0.5*\th:2*\r) -- (\x, 3*\x) -- (0.5*\th:3*\r) -- (3*\x, \x) -- cycle
;

}

\foreach \a in {0,...,3} {

\tikzset{rotate=\a*\th}

\draw[
]
	(1.5*\th:\r) -- (0,2*\x)
	(0,2*\x) -- (0.5*\th:2*\r)
	(0,2*\x) -- (-\x, 3*\x)
	(0,4*\x) -- (\x, 3*\x)
	(0,4*\x) -- (0.5*\th:3.5*\r) 
	(0,4*\x) -- (1.5*\th:3*\r) 
;

\draw[]
	(0.5*\th:\r) -- (1.5*\th:\r)
	(0.5*\th:\r) -- (0.5*\th:2*\r) 
	(0.5*\th:2*\r) -- (\x, 3*\x)
	(1.5*\th:2*\r) -- (-\x, 3*\x)
	(-\x, 3*\x) -- (\x, 3*\x)
	(\x, 3*\x) -- (0.5*\th:3*\r) 
	(-\x, 3*\x) -- (1.5*\th:3*\r) 
	(0.5*\th:3*\r) -- (0.5*\th:3.5*\r) 
;

\node at (0.15*\x,2.25*\x) {\scalebox{0.5}{$\diamond$}}; 
\node at (0.9*\x,2.7*\x) {\scalebox{0.5}{$\diamond$}}; 

\node at (0.15*\x,1.75*\x) {\scalebox{0.5}{$\diamond$}}; 
\node at (0.9*\x,1.25*\x) {\scalebox{0.5}{$\diamond$}}; 

\node at (-0.4*\x,2*\x) {\scalebox{0.5}{$\diamond$}}; 
\node at (-1.6*\x,2*\x) {\scalebox{0.5}{$\diamond$}}; 

\node at (0,3.65*\x) {\scalebox{0.5}{$\diamond$}}; 
\node at (-\x,3.25*\x) {\scalebox{0.5}{$\diamond$}}; 

\node at (0.85*\x,3.55*\x) {\scalebox{0.5}{$\diamond$}}; 
\node at (2.75*\x,3.25*\x) {\scalebox{0.5}{$\diamond$}}; 

\node at (0,4.25*\x) {\scalebox{0.5}{$\diamond$}}; 
\node at (-3.1*\x,3.45*\x) {\scalebox{0.5}{$\diamond$}}; 

}

\draw[] (0,0) circle (3.5*\r);

\end{scope}

\end{tikzpicture}
\caption{The four Platonic and Archimedean type tilings, $\diamond=\beta$}
\label{Fig-a4-platonic-archimedean-tilings}
\end{figure}


\begin{figure}[h!]
\centering
\begin{tikzpicture}[scale=1]

\tikzmath{
\s=1;
\r=0.35;
\th=360/4;
\x=\r*cos(0.5*\th);
\R = sqrt(\x^2+(3*\x)^2);
\aR = acos(3*\x/\R);
\RRR=sqrt( (3*\x)^2 + (4*\x)^2 );
\aRRR=acos(3*\x/\RRR);
}

\begin{scope}[] 

\fill[gray!50] (0,0) circle (3.5*\r);

\fill[white]
	(3*\x, 3*\x) -- (-3*\x, 3*\x) -- (-3*\x, -3*\x) -- (3*\x, -3*\x) -- cycle
;

\fill[gray!50]
	(0.5*\th:\r) -- (1.5*\th:\r) -- (2.5*\th:\r) -- (3.5*\th:\r) -- cycle
;

\foreach \a in {0,...,3} {

\tikzset{shift={(\th+\a*\th:2*\x)}}

\fill[gray!50]
	(0.5*\th:\r) -- (1.5*\th:\r) -- (2.5*\th:\r) -- (3.5*\th:\r) -- cycle
;

}

\foreach \a in {0,...,3}{

\draw[rotate=\th*\a]
	(0.5*\th:\r) -- (1.5*\th:\r)
	(\x, \x) -- (\x, 3*\x)
	(-\x, \x) -- (-\x, 3*\x)
	(-\x, 3*\x) -- (\x, 3*\x)
	(-\x, 3*\x) -- (-3*\x, 3*\x) 
	(\x, 3*\x) -- (3*\x, 3*\x) 
;

\node at (60+\th*\a:2*\r) {\scalebox{0.5}{$\diamond$}};
\node at (30+\th*\a:2*\r) {\scalebox{0.5}{$\diamond$}};

}

\draw (0,0) circle (3*\r);

\end{scope}

\begin{scope}[xshift=3.5*\s cm] 

\fill[gray!50]
	(5*\x, 3*\x) arc (90-\aRRR:90+\aRRR:1.25*\RRR) -- (-5*\x, 3*\x) -- (-5*\x, -3*\x) arc (270-\aRRR:270+\aRRR:1.25*\RRR) -- cycle
;

\fill[white]
	(4*\x,3*\x) -- (-4*\x,3*\x) -- (-4*\x,-3*\x) -- (4*\x,-3*\x)
; 

\foreach \aa in {0,1} {

\fill[gray!50, rotate=2*\aa*\th, shift={(\x,0)}]
	(0.5*\th:\r) -- (1.5*\th:\r) -- (2.5*\th:\r) -- (3.5*\th:\r) -- cycle
;

\fill[gray!50, rotate=2*\aa*\th, shift={(\x,-2*\x)}]
	(0.5*\th:\r) -- (1.5*\th:\r) -- (2.5*\th:\r) -- (3.5*\th:\r) -- cycle
;

\fill[white, rotate=2*\aa*\th, shift={(3*\x,-2*\x)}]
	(0.5*\th:\r) -- (1.5*\th:\r) -- (2.5*\th:\r) -- (3.5*\th:\r) -- cycle
;

\fill[gray!50, rotate=2*\aa*\th]
	(2*\x,\x) -- (2*\x,-\x) -- (4*\x,-\x) -- (4*\x,3*\x) -- cycle
;

}

\foreach \aa in {-1,1} {

\tikzset{shift={(\aa*\x,0)}}

\foreach \a in {0,...,3}{

\draw[rotate=\th*\a]
	(0.5*\th:\r) -- (1.5*\th:\r)
;

}

}

\foreach \aa in {0,1} {

\draw[rotate=2*\aa*\th]
	(0, \x) -- (0, 3*\x) 
	(0, 3*\x) -- (4*\x, 3*\x)
	(2*\x, \x) -- (4*\x, 3*\x)
	(2*\x, -\x) -- (4*\x, -\x)
	(4*\x, -\x) -- (4*\x, 3*\x)
	(0, -\x) -- (0, -3*\x) 
	(2*\x, -\x) -- (2*\x, -3*\x)
	(0, -3*\x) -- (2*\x, -3*\x)
	(2*\x, -3*\x) -- (4*\x, -3*\x)
	(4*\x, -\x) -- (4*\x, -3*\x)
	(4*\x, -3*\x) --  (5*\x, -3*\x)
	(4*\x, 3*\x) arc (90-\aRRR:90+\aRRR:\RRR)
;

{
\tikzset{rotate=\aa*180}

\node at (0.4*\x,2.5*\x) {\scalebox{0.5}{$\diamond$}};
\node at (1.9*\x,1.4*\x) {\scalebox{0.5}{$\diamond$}};

}

{
\foreach \a in {0,1} {

\tikzset{rotate=2*\aa*\th, shift={(3*\x,-2*\x)}}

\node at (0.5*\th+2*\th*\a: 0.6*\r) {\scalebox{0.5}{$\diamond$}}; 

}
}

}

\end{scope}

\end{tikzpicture}
\caption{The two sporadic tilings, $\diamond=\beta$}
\label{Fig-a4-sporadic-tilings}
\end{figure}
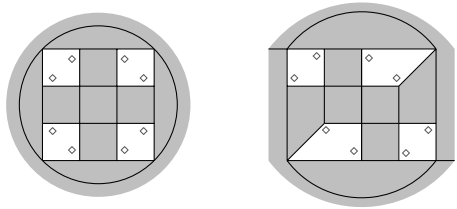

The first tiling in Figure \ref{Fig-a4-platonic-archimedean-tilings} is the (deformed) cube. The second and the third tilings are two triangular fusions of the snub cube. The fourth is a quadrilateral subdivision of the truncated octahedron.

The two triangular fusions of the snub cube are explained in Figure \ref{Fig-a4-Tilings-tri-fusion-snub-cube}.
The snub cube, a dihedral tiling of the sphere by regular triangles and squares, is given in the first picture. In general, if the triangles in a dihedral tiling can be grouped into adjacent pairs, then by fusing all adjacent pairs we get  quadrilaterals. In the snub cube, all triangles are regular. Therefore the fusion produces congruent rhombi. The dashed lines in Figure \ref{Fig-a4-Tilings-tri-fusion-snub-cube} represent the choices of adjacent pairs. The second, third and fourth picture are all possible ways to group the pairs.


\begin{figure}[h!]
\centering
\begin{tikzpicture}[scale=1]

\tikzmath{
\s = 1;
\r= 0.375;
\th=360/4;
\ph=360/8;
\x=\r*cos(0.5*\th);
}

\begin{scope}[] 

\fill[gray!50] (0,0) circle (4.3*\r);
	
\fill[white] (0,0) circle (4*\r);

\fill[gray!50]
	(0.5*\th:\r) -- (1.5*\th:\r) -- (2.5*\th:\r) -- (3.5*\th:\r) -- cycle
;

\foreach \a in {0,...,3} {

\fill[gray!50, rotate=\a*\th]
	(90:2*\r) -- (90:3*\r) -- (90-0.5*\th:3*\r) -- (90-0.5*\th:2*\r) -- cycle
;

}

\foreach \a in {0,...,3} {

\draw[rotate=\a*\th]
	(0.5*\th:\r) -- (1.5*\th:\r)
	(1.5*\th:\r) -- (90:2*\r)
	(0.5*\th:\r) -- (90-\ph:2*\r)
	(90:3*\r) -- (90:4*\r)
	(90+0.5*\th:3*\r) -- (90:4*\r)
;


}

\foreach \a in {1,3,...,7} {

\draw[rotate=\a*\ph]
	(90:2*\r) -- (90+\ph:2*\r)
	(90:3*\r) -- (90+\ph:3*\r)
;

}

\foreach \a in {0,2,...,6} {

\draw[rotate=\a*\ph]
	(90+0.5*\th:2*\r) -- (90:3*\r)
;

\draw[
rotate=\a*\ph]
	(90:2*\r) -- (90-0.5*\th:\r)
	(90:2*\r) -- (90+\ph:2*\r)
	(90:3*\r) -- (90+\ph:3*\r)
	(90-0.5*\th:3*\r) -- (90:4*\r)
;

}

\foreach \a in {0,...,7} {

\draw[rotate=\a*\ph]
	(90:2*\r) -- (90:3*\r)
;

}

\draw[] (0,0) circle (4*\r);

\end{scope} 

\begin{scope}[xshift=3.5*\s cm] 

\fill[gray!50] (0,0) circle (4.3*\r);
	
\fill[white] (0,0) circle (4*\r);

\fill[gray!50]
	(0.5*\th:\r) -- (1.5*\th:\r) -- (2.5*\th:\r) -- (3.5*\th:\r) -- cycle
;

\foreach \a in {0,...,3} {

\fill[gray!50, rotate=\a*\th]
	(90:2*\r) -- (90:3*\r) -- (90-0.5*\th:3*\r) -- (90-0.5*\th:2*\r) -- cycle
;

}

\foreach \a in {0,...,3} {

\draw[rotate=\a*\th]
	(0.5*\th:\r) -- (1.5*\th:\r)
	(1.5*\th:\r) -- (90:2*\r)
	(0.5*\th:\r) -- (90-\ph:2*\r)
	(90:3*\r) -- (90:4*\r)
	(90+0.5*\th:3*\r) -- (90:4*\r)
;

\node [shape = circle, fill = black, minimum size = 0.14 cm, inner sep=0pt] at (0.5*\th+\a*\th:\r) {};
\node [shape = circle, fill = black, minimum size = 0.14 cm, inner sep=0pt] at (90+\a*\th: 4*\r) {};

}

\foreach \a in {1,3,...,7} {

\draw[rotate=\a*\ph]
	(90:2*\r) -- (90+\ph:2*\r)
	(90:3*\r) -- (90+\ph:3*\r)
;

}

\foreach \a in {0,2,...,6} {

\draw[rotate=\a*\ph]
	(90+0.5*\th:2*\r) -- (90:3*\r)
;

\draw[gray, densely dashed, rotate=\a*\ph]
	(90:2*\r) -- (90-0.5*\th:\r)
	(90:2*\r) -- (90+\ph:2*\r)
	(90:3*\r) -- (90+\ph:3*\r)
	(90-0.5*\th:3*\r) -- (90:4*\r)
;

}

\foreach \a in {0,...,7} {

\draw[rotate=\a*\ph]
	(90:2*\r) -- (90:3*\r)
;

}

\draw[] (0,0) circle (4*\r);

\end{scope} 

\begin{scope}[xshift=7*\s cm] 

\fill[gray!50] (0,0) circle (4.3*\r);
	
\fill[white] (0,0) circle (4*\r);

\fill[gray!50]
	(0.5*\th:\r) -- (1.5*\th:\r) -- (2.5*\th:\r) -- (3.5*\th:\r) -- cycle
;

\foreach \a in {0,...,3} {

\fill[gray!50, rotate=\a*\th]
	(90:2*\r) -- (90:3*\r) -- (90+0.5*\th:3*\r) -- (90+0.5*\th:2*\r) -- cycle
;

}

\foreach \a in {0,...,3} {

\draw[rotate=\a*\th]
	(0.5*\th:\r) -- (1.5*\th:\r)
	%
	%
;

}

\foreach \a in {0,...,7} {

\draw[rotate=\a*\ph]
	(90:2*\r) -- (90+\ph:2*\r)
	(90:2*\r) -- (90:3*\r)
	(90:3*\r) -- (90+\ph:3*\r)
;

}

\foreach \a in {0,1} {

\draw[rotate=\a*180]
	(0.5*\th:\r) -- (0.5*\th:2*\r)
	(1.5*\th:\r) -- (90:2*\r)
	(1.5*\th:\r) -- (180:2*\r)
	(90:3*\r) -- (90:4*\r)
	(0:4*\r) -- (0.5*\th:3*\r) 
	(0:4*\r) -- (-0.5*\th:3*\r) 
;

\draw[gray, densely dashed, rotate=\a*180]
	(0.5*\th:\r) -- (0:2*\r)
	(0.5*\th:\r) -- (90:2*\r)
	(1.5*\th:\r) -- (1.5*\th:2*\r)
	(0.5*\th:2*\r) -- (90:3*\r)
	(1.5*\th:2*\r) -- (180:3*\r)
	(0:3*\r) -- (0:4*\r)
	(90:4*\r) -- (1.5*\th:3*\r) 
	(90:4*\r) -- (0.5*\th:3*\r) 
;

\node [shape = circle, fill = black, minimum size = 0.14 cm, inner sep=0pt] at (0.5*\th+\a*180:2*\r) {};
\node [shape = circle, fill = black, minimum size = 0.14 cm, inner sep=0pt] at (0+\a*180: 2*\r) {};

\node [shape = circle, fill = black, minimum size = 0.14 cm, inner sep=0pt] at (90+\a*180: 3*\r) {};
\node [shape = circle, fill = black, minimum size = 0.14 cm, inner sep=0pt] at (1.5*\th+\a*180: 3*\r) {};

}

\draw[] (0,0) circle (4*\r);

\end{scope}

\begin{scope}[xshift=10.5*\s cm] 

\fill[gray!50] (0,0) circle (4.3*\r);
	
\fill[white] (0,0) circle (4*\r);

\fill[gray!50]
	(0.5*\th:\r) -- (1.5*\th:\r) -- (2.5*\th:\r) -- (3.5*\th:\r) -- cycle
;

\foreach \a in {0,...,3} {

\fill[gray!50, rotate=\a*\th]
	(90:2*\r) -- (90:3*\r) -- (90+0.5*\th:3*\r) -- (90+0.5*\th:2*\r) -- cycle
;

}

\foreach \a in {0,...,3} {

\draw[rotate=\a*\th]
	(0.5*\th:\r) -- (1.5*\th:\r)
	%
	%
;

}

\foreach \a in {0,...,7} {

\draw[rotate=\a*\ph]
	(90:2*\r) -- (90+\ph:2*\r)
	(90:2*\r) -- (90:3*\r)
	(90:3*\r) -- (90+\ph:3*\r)
;

}

\foreach \a in {0,1} {

\draw[rotate=\a*180]
	(0.5*\th:\r) -- (0.5*\th:2*\r)
	(1.5*\th:\r) -- (90:2*\r)
	(1.5*\th:\r) -- (180:2*\r)
	%
	(0:3*\r) -- (0:4*\r)
	(90:4*\r) -- (1.5*\th:3*\r) 
	(90:4*\r) -- (0.5*\th:3*\r) 
;

\draw[gray, densely dashed, rotate=\a*180]
	(0.5*\th:\r) -- (0:2*\r)
	(0.5*\th:\r) -- (90:2*\r)
	(1.5*\th:\r) -- (1.5*\th:2*\r)
	(0.5*\th:2*\r) -- (90:3*\r)
	(1.5*\th:2*\r) -- (180:3*\r)
	%
	(90:3*\r) -- (90:4*\r)
	(0:4*\r) -- (0.5*\th:3*\r) 
	(0:4*\r) -- (-0.5*\th:3*\r) 
;

\node [shape = circle, fill = black, minimum size = 0.14 cm, inner sep=0pt] at (0.5*\th+\a*180:2*\r) {};
\node [shape = circle, fill = black, minimum size = 0.14 cm, inner sep=0pt] at (0.5*\th+\a*180:3*\r) {};

\node [shape = circle, fill = black, minimum size = 0.14 cm, inner sep=0pt] at (0+\a*180: 2*\r) {};
\node [shape = circle, fill = black, minimum size = 0.14 cm, inner sep=0pt] at (0+\a*180: 3*\r) {};

}

\draw[] (0,0) circle (4*\r);

\end{scope}

\end{tikzpicture}
\caption{Triangular fusions of the snub cube}
\label{Fig-a4-Tilings-tri-fusion-snub-cube}
\end{figure}
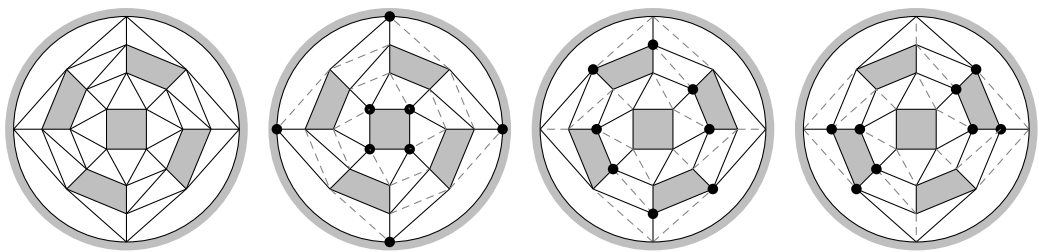 

To distinguish the last three tilings in Figure \ref{Fig-a4-Tilings-tri-fusion-snub-cube}, we highlight with $\bullet$'s the vertices at the squares with exterior edge configuration \quotes{dashed-solid-solid}. The distribution of $\bullet$'s shows that the the second and the third tiling in Figure \ref{Fig-a4-platonic-archimedean-tilings} are different. Guided by the distribution of $\bullet$'s, the second and the fourth tiling are equivalent. Hence there are actually two tilings up to equivalence.

The two tilings are related in the following way. In the third tiling, by cutting along the octagonal boundary (highlighted in Figure \ref{Fig-a4-tri-fusion-modification}) of the ring of six rhombi surrounding the exterior square, we get the inner part and the outer part. Rotating the outer part, we obtain the tiling in the fourth picture with two squares having four $\bullet$'s. We will see in Proposition \ref{Prop-albe2} that a square with four such vertices uniquely determine the tiling in the second picture. This further shows how one can derive one tiling from another through a specific cut-and-rotate operation (or modification \cite{cl2, cly, ly}).

\begin{figure}[h!]
\centering
\pgfmathsetmacro{\r}{2} %
\tdplotsetmaincoords{75}{110}
\begin{tikzpicture}[]

\tikzmath{
\s = 1;
}

\begin{scope}[tdplot_main_coords,
font=\footnotesize,
Sphericalcap/.style={gray!70!black,
}
]

\tikzmath{
\r=1.6;
\R=0.35*\r; 
\v=sqrt(3);
}

\pgfmathsetmacro{\h}{\R} %
\coordinate[] (M) at (0,0,0); 
\coordinate[] (N) at (0,0,\h); 
\coordinate[] (S) at (0,0,-\h); 

\node at (0,0,-\h) {\AxisRotator[rotate=-90, black!60]};

\draw[Sphericalcap, teal!80!blue, line width=1] (N) circle[radius=sqrt(\r*\r-\h*\h)];
\draw[Sphericalcap, red!70!white, line width=1] (S) circle[radius=sqrt(\r*\r-\h*\h)];

\begin{scope}[tdplot_screen_coords, on background layer]
\fill[ball color= gray!5, opacity = 0.1] (M) circle (\r); 
\end{scope}

\end{scope}

\begin{scope}[xshift=4*\s cm] 

\tikzmath{
\r= 0.375;
\th=360/4;
\ph=360/8;
\x=\r*cos(0.5*\th);
}

\fill[gray!50] (0,0) circle (4.3*\r);
	
\fill[white] (0,0) circle (4*\r);

\fill[gray!50]
	(0.5*\th:\r) -- (1.5*\th:\r) -- (2.5*\th:\r) -- (3.5*\th:\r) -- cycle
;

\foreach \a in {0,...,3} {

\fill[gray!50, rotate=\a*\th]
	(90:2*\r) -- (90:3*\r) -- (90+0.5*\th:3*\r) -- (90+0.5*\th:2*\r) -- cycle
;

}

\foreach \a in {0,...,3} {

\draw[rotate=\a*\th]
	(0.5*\th:\r) -- (1.5*\th:\r)
	%
	%
;

}

\foreach \a in {0,...,7} {

\draw[rotate=\a*\ph]
	(90:2*\r) -- (90:3*\r)
;

}

\foreach \a in {0,1} {

\draw[rotate=\a*180]
	(0.5*\th:\r) -- (0.5*\th:2*\r)
	(1.5*\th:\r) -- (90:2*\r)
	(1.5*\th:\r) -- (180:2*\r)
	(90:3*\r) -- (90:4*\r)
	(0:4*\r) -- (0.5*\th:3*\r) 
	(0:4*\r) -- (-0.5*\th:3*\r) 
;

\draw[gray, densely dashed, rotate=\a*180]
	(0.5*\th:\r) -- (0:2*\r)
	(0.5*\th:\r) -- (90:2*\r)
	(1.5*\th:\r) -- (1.5*\th:2*\r)
	(0.5*\th:2*\r) -- (90:3*\r)
	(1.5*\th:2*\r) -- (180:3*\r)
	(0:3*\r) -- (0:4*\r)
	(90:4*\r) -- (1.5*\th:3*\r) 
	(90:4*\r) -- (0.5*\th:3*\r) 
;

}

\foreach \a in {0,...,7} {

\draw[teal!80!blue, line width=1, rotate=\a*\ph]
	(90:2*\r) -- (90+\ph:2*\r)
;

\draw[red!70!white, line width=1, rotate=\a*\ph]
	(90:3*\r) -- (90+\ph:3*\r)
;

}

\foreach \a in {0,1} {

\node [shape = circle, fill = black, minimum size = 0.14 cm, inner sep=0pt] at (0.5*\th+\a*180:2*\r) {};
\node [shape = circle, fill = black, minimum size = 0.14 cm, inner sep=0pt] at (0+\a*180: 2*\r) {};

\node [shape = circle, fill = black, minimum size = 0.14 cm, inner sep=0pt] at (90+\a*180: 3*\r) {};
\node [shape = circle, fill = black, minimum size = 0.14 cm, inner sep=0pt] at (1.5*\th+\a*180: 3*\r) {};

}

\draw[] (0,0) circle (4*\r);

\end{scope}

\begin{scope}[xshift=8*\s cm] 

\tikzmath{
\r= 0.375;
\th=360/4;
\ph=360/8;
\x=\r*cos(0.5*\th);
}

\fill[gray!50] (0,0) circle (4.3*\r);
	
\fill[white] (0,0) circle (4*\r);

\fill[gray!50]
	(0.5*\th:\r) -- (1.5*\th:\r) -- (2.5*\th:\r) -- (3.5*\th:\r) -- cycle
;

\foreach \a in {0,...,3} {

\fill[gray!50, rotate=\a*\th]
	(90:2*\r) -- (90:3*\r) -- (90+0.5*\th:3*\r) -- (90+0.5*\th:2*\r) -- cycle
;

}

\foreach \a in {0,...,3} {

\draw[rotate=\a*\th]
	(0.5*\th:\r) -- (1.5*\th:\r)
	%
	%
;

}

\foreach \a in {0,...,7} {

\draw[rotate=\a*\ph]
	(90:2*\r) -- (90:3*\r)
;

}

\foreach \a in {0,1} {

\draw[rotate=\a*180]
	(0.5*\th:\r) -- (0.5*\th:2*\r)
	(1.5*\th:\r) -- (90:2*\r)
	(1.5*\th:\r) -- (180:2*\r)
	%
	(0:3*\r) -- (0:4*\r)
	(90:4*\r) -- (1.5*\th:3*\r) 
	(90:4*\r) -- (0.5*\th:3*\r) 
;

\draw[gray, densely dashed, rotate=\a*180]
	(0.5*\th:\r) -- (0:2*\r)
	(0.5*\th:\r) -- (90:2*\r)
	(1.5*\th:\r) -- (1.5*\th:2*\r)
	(0.5*\th:2*\r) -- (90:3*\r)
	(1.5*\th:2*\r) -- (180:3*\r)
	%
	(90:3*\r) -- (90:4*\r)
	(0:4*\r) -- (0.5*\th:3*\r) 
	(0:4*\r) -- (-0.5*\th:3*\r) 
;

}

\foreach \a in {0,...,7} {

\draw[rotate=\a*\ph]
	(90:2*\r) -- (90:3*\r)
;

}

\foreach \a in {0,...,7} {

\draw[teal!80!blue, line width=1, rotate=\a*\ph]
	(90:2*\r) -- (90+\ph:2*\r)
;

\draw[red!70!white, line width=1, rotate=\a*\ph]
	(90:3*\r) -- (90+\ph:3*\r)
;

}

\foreach \a in {0,1} {

\node [shape = circle, fill = black, minimum size = 0.14 cm, inner sep=0pt] at (0.5*\th+\a*180:2*\r) {};
\node [shape = circle, fill = black, minimum size = 0.14 cm, inner sep=0pt] at (0.5*\th+\a*180:3*\r) {};

\node [shape = circle, fill = black, minimum size = 0.14 cm, inner sep=0pt] at (0+\a*180: 2*\r) {};
\node [shape = circle, fill = black, minimum size = 0.14 cm, inner sep=0pt] at (0+\a*180: 3*\r) {};

}

\draw[] (0,0) circle (4*\r);

\end{scope}
\end{tikzpicture}
\caption{The modification of the two triangular fusions}
\label{Fig-a4-tri-fusion-modification}
\end{figure}
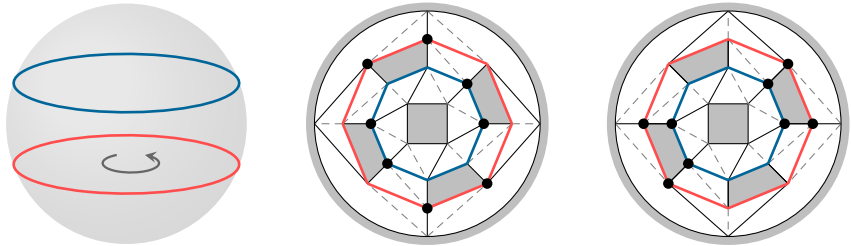

The fourth tiling of Figure \ref{Fig-a4-platonic-archimedean-tilings} results from a quadrilateral subdivision of the hexagons in the truncated octahedron. The subdivision is illustrated in Figure \ref{Fig-a4-Tilings-quad-subdiv-trunc-octahedron}, where each hexagon is subdivided into three rhombi and the corresponding degree $3$ vertex at \quotes{its centre} is $\beta^3$.

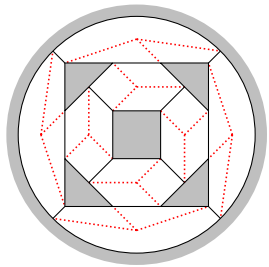
\begin{figure}[h!]
\centering
\begin{tikzpicture}[scale=1]

\tikzmath{
\s = 1;
\r= 0.45;
\th=360/4;
\ph=360/8;
\x=\r*cos(0.5*\th);
}

\fill[gray!50] (0,0) circle (3.85*\r);

\fill[white] (0,0) circle (3.5*\r);

\fill[gray!50]
	(0.5*\th:\r) -- (1.5*\th:\r) -- (2.5*\th:\r) -- (3.5*\th:\r) -- cycle
;

\foreach \a in {0,1,2,3} {

\fill[gray!50, rotate=\a*\th]
	(0.5*\th:2*\r) -- (\x, 3*\x) -- (0.5*\th:3*\r) -- (3*\x, \x) -- cycle
;

}

\foreach \a in {0,...,3} {

\draw[red, densely dotted, line width=0.6, rotate=\a*\th]
	(1.5*\th:\r) -- (0,2*\x)
	(0,2*\x) -- (0.5*\th:2*\r)
	(0,2*\x) -- (-\x, 3*\x)
	(0,4*\x) -- (\x, 3*\x)
	(0,4*\x) -- (0.5*\th:3.5*\r) 
	(0,4*\x) -- (1.5*\th:3*\r) 
;

\draw[rotate=\a*\th]
	(0.5*\th:\r) -- (1.5*\th:\r)
	(0.5*\th:\r) -- (0.5*\th:2*\r) 
	(0.5*\th:2*\r) -- (\x, 3*\x)
	(1.5*\th:2*\r) -- (-\x, 3*\x)
	(-\x, 3*\x) -- (\x, 3*\x)
	(\x, 3*\x) -- (0.5*\th:3*\r) 
	(-\x, 3*\x) -- (1.5*\th:3*\r) 
	(0.5*\th:3*\r) -- (0.5*\th:3.5*\r) 
;

}

\draw[] (0,0) circle (3.5*\r);

\end{tikzpicture}
\caption{The quadrilateral subdivision of the truncated octahedron} 
\label{Fig-a4-Tilings-quad-subdiv-trunc-octahedron}
\end{figure}

This paper is organised as follows. In Section \ref{Sec-basic}, we explain the basic terminology and tools. In Section \ref{Sec-tilings}, we classify the tilings. The key to classification is to find all the vertices in the tilings. The sporadic tilings are obtained in Propositions \ref{Prop-albega}, \ref{Prop-be3}, \ref{Prop-albe2}, \ref{Prop-al2be}. The infinite family is obtained in Proposition \ref{Prop-be2ga}.

\section{Basics} \label{Sec-basic}

We first list the basic combinatorial facts of tilings of the sphere by quadrilaterals from \cite{cly},
\begin{align}
\label{Eq-quad-v3}
v_3 &= 8 + \sum_{h\ge4} (h - 4) v_h, \\
\label{Eq-quad-f}
f &= 6 + \sum_{h\ge4} (h-3)v_h,
\end{align}
where $f$ denotes the total number of tiles and $v_i$ denotes the number of degree $i$ vertices and $i \ge 3$. By \eqref{Eq-quad-v3}, there is a degree $3$ vertex. 

The {\em vertex angle sum} of a vertex $\alpha^a\beta^b\gamma^c$, consisting of $a$ copies of $\alpha$ and $b$ copies of $\beta$ and $c$ copies of $\gamma$, is
\begin{align}\label{Eq-vertex-angle-sum}
a \alpha + b \beta + c \gamma = 2\pi.
\end{align} 
In a vertex notation, $a,b,c$ are assumed to be $>0$ unless otherwise specified. That is, we express only the angles appearing at a vertex as a convention. For example, $\alpha\beta^2$ is a vertex with $a=1, b=2$ and $c=0$. The notation $\alpha\beta^2\cdots$ means a vertex with at least one $\alpha$ and two $\beta$'s, i.e., $a\ge1$ and $b\ge2$. The angle combination in $\cdots$ is called the {\em remainder} of the vertex. The value of the remainder is denoted by $R$. For example, $R(\alpha\beta^2) = 2\pi - \alpha - 2\beta$.

There are various constraints on the angle combinations at vertices in a tiling. Examples of such constraints are the vertex angle sum and the quadrilateral angle sum. A collection of all vertices in a tiling satisfying various constraints is called an {\em anglewise vertex combination} ($\AVC$). The following $\AVC$ is from \eqref{Eq-AVC-be2ga-begac-albegac},
\begin{align*}
\AVC = \{ \beta^2\gamma, \beta\gamma^c, \alpha\beta\gamma^c \}.
\end{align*}
The generic $c$ may take different values at different vertex. The tilings constructed from \eqref{Eq-AVC-be2ga-begac-albegac} do not have $\beta\gamma^c$. We use \quotes{$\equiv$} in place of \quotes{$=$} to denote the set of all vertices which actually appear in a tiling. For example, we have \eqref{Eq-AVC-be2ga-albegac} for the tilings in Figure \ref{Fig-a4-Tilings-be2ga-albega2}
\begin{align*}
\AVC \equiv \{ \beta^2\gamma, \alpha\beta\gamma^c \}.
\end{align*}

To obtain the vertices, it is convenient to have notations for studying various angle arrangements. For example, $\alpha_1\gamma_2\cdots$ denotes the vertex where $T_1$ contributes $\alpha$ and $T_2$ contributes $\gamma$ in the first picture of Figure \ref{Fig-a4-adj-square-rhombus}. To emphasize $\alpha_1$ is adjacent to $\gamma_2$ along an edge \quotes{ $\vert$ }, we use $\alpha_1 \vert \gamma_2 \cdots$ to denote the vertex. In addition, the same picture shows that $\alpha\vert\gamma\cdots$ is a vertex if and only if $\alpha\vert\beta\cdots$ is a vertex. Similarly, $T_1, T_2$ in the second picture of Figure \ref{Fig-a4-adj-square-rhombus} show that $\gamma\vert\gamma\cdots$ is a vertex if and only if $\beta\vert\beta\cdots$ is also a vertex. For a full vertex, such as $\alpha^3$ in the third picture, we use $\vert \alpha \vert \alpha \vert \alpha \vert$ to denote its angle arrangement.



\begin{figure}[h!] 
\centering
\begin{tikzpicture}

\tikzmath{
\s=1;
\r=0.8;
\th=360/4;
\x=\r*cos(0.5*\th);
\R=sqrt(\x^2+(3*\x)^2);
\aR=acos(3*\x/\R);
}

\begin{scope}[] 

\foreach \aa in {0,1} {

\begin{scope}[xshift=2*\aa*\x cm]

\foreach \a in {0,...,3} {

\draw[rotate=\th*\a]
	(0.5*\th:\r) -- (1.5*\th:\r)
;

}

\end{scope}

}

\foreach \a in {0,...,3} {

\node at (0.5*\th+\th*\a: 0.65*\r) {\small $\alpha$};

}

\foreach \aa in {1} {

\begin{scope}[xshift=2*\aa*\x cm, xscale=\aa]

\foreach \a in {0,2} {

\node at (1.5*\th+\th*\a: 0.625*\r) {\small $\beta$}; 
\node at (0.5*\th+\th*\a: 0.65*\r) {\small $\gamma$}; 

}

\end{scope}

}

\node[inner sep=1,draw,shape=circle] at (0:0) {\small $1$};
\node[inner sep=1,draw,shape=circle] at (2*\x,0) {\small $2$};

\end{scope} 

\begin{scope}[xshift=4*\s cm] 

\foreach \aa in {-1,1} {

\tikzset{shift={(\aa*\x,0)}, xscale=\aa}

\foreach \a in {0,...,3} {

\draw[rotate=\th*\a]
	(0.5*\th:\r) -- (1.5*\th:\r)
;

}

\foreach \a in {0,2} {

\node at (1.5*\th+\th*\a: 0.625*\r) {\small $\beta$}; 
\node at (0.5*\th+\th*\a: 0.65*\r) {\small $\gamma$}; 

}

}

\node[inner sep=1,draw,shape=circle] at (-\x,0) {\small $1$};
\node[inner sep=1,draw,shape=circle] at (\x,0) {\small $2$};

\end{scope} 

\begin{scope}[xshift=7*\s cm, yshift=-0.1*\s cm] 

\foreach \a in {0,1,2} {

\draw[rotate=120*\a]
	(0:0) -- (90:\r) 
;

\node at (270+\a*120:0.3*\r) {\small $\alpha$};

}

\end{scope} 

\end{tikzpicture}
\caption{The arrangements of $\alpha \vert \gamma$ and $\gamma\vert\gamma$ and $\beta\vert\beta$ and $\alpha^3$}
\label{Fig-a4-adj-square-rhombus}
\end{figure}
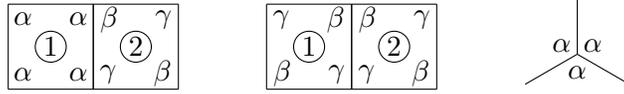


The tiles are the square and the rhombus in Figure \ref{Fig-quad-a4-angles}. Up to symmetry, we may assume $\beta > \gamma$ in the rhombus. The assumption is implicit throughout this paper.

The angle sum of the four angles of a spherical quadrilateral is $>2\pi$. For the square we have $4\alpha>2\pi$ and for the rhombus we have $2\beta+2\gamma>2\pi$. Then we get
\begin{align*}
\alpha > \tfrac{1}{2}\pi, \quad
\beta+\gamma > \pi.
\end{align*}

\begin{lem}\label{Lem-albe-alga} In a dihedral tiling by the square and the rhombus, $\alpha\beta\cdots, \alpha\gamma\cdots$ are vertices.
\end{lem}

\begin{proof} The assertion is immediate from the first picture of Figure \ref{Fig-a4-adj-square-rhombus}.
\end{proof}

An immediate consequence of Lemma \ref{Lem-albe-alga} is that one of $\alpha^a\beta^b, \alpha\beta\gamma\cdots$ is a vertex. Then one of the inequalities $2\alpha + \beta \le 2\pi$ and $\alpha+2\beta \le 2\pi$ and $\alpha+\beta+\gamma\le2\pi$ holds. Combined with $\alpha>\frac{1}{2}\pi$ and $\beta>\gamma$ and $\beta+\gamma>\pi$, we know $\alpha, \gamma<\pi$.

\begin{lem}\label{Lem-al-be-ga} For the square and the rhombus with the same edges, we have 
\begin{align}\label{Eq-tan-al-be-ga}
\tan^2 \tfrac{1}{2} \alpha = \tan \tfrac{1}{2}\beta \tan \tfrac{1}{2} \gamma.
\end{align}
For $\beta>\gamma$, we have 
\begin{align}\label{Eq-Angle-Ineq}
\pi > \beta > \alpha > \gamma.
\end{align}
\end{lem}

We remark that $\cos x = \cot^2 \tfrac{1}{2} \alpha = \cot \tfrac{1}{2}\beta \cot \tfrac{1}{2} \gamma$.

\begin{proof} By \cite[Art. 62]{to} or \cite[Lemma 18]{cly}, both sides of \eqref{Eq-tan-al-be-ga} is $\tfrac{1}{\cos x}$. By $\alpha,\gamma< \pi$, it implies $\tan \frac{1}{2} \beta> 0$. Then we have $\beta <\pi$. By $\beta>\gamma$ and $\tan t$ is strictly increasing on $(0,\frac{1}{2}\pi)$, it further implies $\pi > \beta > \alpha > \gamma$.
\end{proof}

From this point onward, we may assume $\beta > \alpha > \gamma$. Then by Lemma \ref{Lem-albe-alga}, the vertex $\alpha\beta\cdots$ implies 
\begin{align} \label{Eq-sum-al-be-ga}
\alpha + \beta + \gamma \le 2\pi.
\end{align}

The following lemma is an adaptation of \cite[Lemma 4]{cly}.

\begin{lem}[Counting Lemma] \label{Lem-Counting} In a dihedral tiling of the sphere by squares and rhombi, if at every vertex the number of $\beta$ is no more than the number of $\gamma$, then at every vertex these two numbers are equal.
\end{lem}


\section{Tilings} \label{Sec-tilings}

We already know that there is a degree $3$ vertex. By \eqref{Eq-Angle-Ineq}, \eqref{Eq-sum-al-be-ga}, and $\alpha\beta\cdots$ being a vertex, we know that $\gamma^3,  \alpha^2\gamma,  \alpha\gamma^2, \beta\gamma^2$ are not vertices. Hence one of the following degree $3$ vertices must appear in a dihedral tiling,
\begin{align*}
\alpha^3, \beta^3, \alpha^2\beta, \alpha\beta^2, \beta^2\gamma, \alpha\beta\gamma.
\end{align*}
We study these vertices in the following propositions. 

\begin{prop}\label{Prop-al3} There is no dihedral tiling with vertex $\alpha^3$.
\end{prop}

\begin{proof} If $\alpha^3$ is a vertex, then we have $\alpha = \frac{2}{3}\pi$. Combined with $\beta>\alpha>\gamma$ and $\beta+\gamma>\pi$, we have $R(\beta^2)<\alpha, 2\gamma$. So $\beta^2\cdots=\beta^2\gamma$. 

Suppose both $\alpha^3, \beta^2\gamma$ are vertices. We get
\begin{align*}
\alpha=\tfrac{2}{3}\pi, \quad
\beta = \pi - \tfrac{1}{2}\gamma. 
\end{align*}
Substituting the above into \eqref{Eq-tan-al-be-ga}, we further get 
\begin{align*}
3 = \cot \tfrac{1}{4}\gamma \tan \tfrac{1}{2} \gamma = \frac{2}{1-\tan^2 \tfrac{1}{4}\gamma},
\end{align*}
which implies $\gamma= \frac{2}{3}\pi \ge \alpha$, a contradiction. So $\beta^2\gamma$ is not a vertex and hence $\beta^2\cdots$ is not a vertex.

By $\beta> \alpha = \frac{2}{3}\pi$ and $\beta+\gamma>\pi$ and no $\beta^2\cdots$, we get $\beta\cdots=\beta\gamma^c, \alpha\beta\gamma^c$. Counting Lemma (Lemma \ref{Lem-Counting}) implies $\beta\cdots=\alpha\beta\gamma$ and hence $\gamma\cdots=\alpha\beta\gamma$. Hence the vertices are
\begin{align}\label{Eq-AVC-al3-albega}
\AVC = \{ \alpha^3, \alpha\beta\gamma \}. 
\end{align}

Start at an $\alpha^3$, by $\alpha^2\cdots=\alpha^3$ we determine a monohedral tiling in the first picture of Figure \ref{Fig-a4-Tiling-f6-cube}. Therefore there is no dihedral tiling for $\alpha^3$.
\end{proof}

\begin{figure}[h!] 
\centering
\begin{tikzpicture}

\tikzmath{
\r=1.2;
\th=360/3;
\x=\r*cos(0.5*\th);
}

\begin{scope}[] 

\foreach \a in {0,1,2} {

\draw[rotate=\th*\a]
	(0:0) -- (90:\r) 
	(90:\r) -- (90-0.5*\th:2*\x)
	(90:\r) -- (90+0.5*\th:2*\x)
	(90-0.5*\th:2*\x) -- (90-0.5*\th:3*\x)
;

\node at (270+\th*\a:0.2*\r) {\small $\alpha$};
\node at (270+\th*\a:0.8*\r) {\small $\alpha$};
\node at (270+0.425*\th+\th*\a:0.8*\r) {\small $\alpha$};
\node at (270-0.425*\th+\th*\a:0.8*\r) {\small $\alpha$};

\node at (90+\th*\a:1.15*\r) {\small $\alpha$};
\node at (90+\th*\a:1.5*\r) {\small $\alpha$};
\node at (37.5+\th*\a:1.1*\r) {\small $\alpha$};
\node at (142.5+\th*\a:1.1*\r) {\small $\alpha$};

}

\end{scope}

\begin{scope}[xshift=4*\r cm] 

\foreach \a in {0,1,2} {

\draw[rotate=\th*\a]
	(0:0) -- (90:\r) 
	(90:\r) -- (90-0.5*\th:2*\x)
	(90:\r) -- (90+0.5*\th:2*\x)
	(90-0.5*\th:2*\x) -- (90-0.5*\th:3*\x)
;

}

\node at (270:0.2*\r) {\small $\alpha$};
\node at (270:0.8*\r) {\small $\alpha$};
\node at (270+0.425*\th:0.8*\r) {\small $\alpha$};
\node at (270-0.425*\th:0.8*\r) {\small $\alpha$};

\node at (150:0.2*\r) {\small $\beta$};
\node at (150:0.8*\r) {\small $\beta$};
\node at (100:0.75*\r) {\small $\gamma$};
\node at (200:0.8*\r) {\small $\gamma$};

\node at (30:0.2*\r) {\small $\gamma$};
\node at (30:0.8*\r) {\small $\gamma$};
\node at (80:0.75*\r) {\small $\beta$};
\node at (-20:0.8*\r) {\small $\beta$};

\node at (90:1.15*\r) {\small $\alpha$};
\node at (90:1.5*\r) {\small $\alpha$};
\node at (37.5:1.1*\r) {\small $\alpha$};
\node at (142.5:1.1*\r) {\small $\alpha$};

\node at (210:1.15*\r) {\small $\beta$};
\node at (210:1.5*\r) {\small $\beta$};
\node at (160:1.1*\r) {\small $\gamma$};
\node at (264:1.15*\r) {\small $\gamma$};

\node at (330:1.15*\r) {\small $\gamma$};
\node at (330:1.5*\r) {\small $\gamma$};
\node at (278:1.15*\r) {\small $\beta$};
\node at (22.5:1.1*\r) {\small $\beta$};

\end{scope}

\end{tikzpicture}
\caption{The tiling with $\alpha^3$ and the tiling with $ \alpha\beta\gamma $}
\label{Fig-a4-Tiling-f6-cube}
\end{figure}
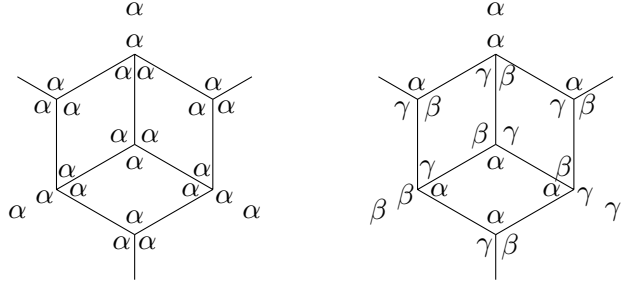

\begin{prop}\label{Prop-albega} The dihedral tiling with vertex $\alpha\beta\gamma$ is the second picture of Figure \ref{Fig-a4-Tiling-f6-cube}.
\end{prop}

The tiling is given by the cube and has $2$ squares and $4$ rhombi. The tiling is the first one in Figure \ref{Fig-a4-platonic-archimedean-tilings}.

\begin{proof} By $\beta>\alpha>\gamma$ and $\alpha\beta\gamma$, we have $\beta>\frac{2}{3}\pi$ and $R(\beta^2)<R(\alpha\beta)=\gamma<\alpha$. Then $\beta^2\cdots$ is not a vertex and $\alpha\beta\cdots=\alpha\beta\gamma$. Hence $\beta\cdots= \alpha\beta\gamma, \beta\gamma^c$. Then Counting Lemma implies $\beta\cdots=\gamma\cdots=\alpha\beta\gamma$. Starting at an $\alpha\beta\gamma$, by $\beta\cdots=\gamma\cdots=\alpha\beta\gamma$, we obtain the second tiling in Figure \ref{Fig-a4-Tiling-f6-cube}. 
\end{proof}

\begin{prop}\label{Prop-be3} The dihedral tiling with vertex $\beta^3$ is the last tiling in Figure \ref{Fig-a4-platonic-archimedean-tilings}.
\end{prop}

The dihedral tiling has $6$ squares and $24$ rhombi. 

\begin{proof} By $\beta^3$, we get $\beta=\frac{2}{3}\pi$. Combined with $\beta>\alpha>\gamma$ and $\beta+\gamma>\pi$, we get the following,
\begin{align*}
\alpha> \tfrac{1}{2}\pi, \quad
\beta=\tfrac{2}{3}\pi, \quad
\gamma>\tfrac{1}{3}\pi.
\end{align*}

Lemma \ref{Lem-albe-alga} implies that $\alpha\beta\cdots$ is a vertex. By $R(\alpha\beta)<\frac{5}{6}\pi<\alpha+\gamma, 3\gamma$, we get $\alpha\beta\cdots=\alpha\beta\gamma^2$. Then $\beta^3, \alpha\beta\gamma^2$ imply
\begin{align*}
\alpha + 2\gamma = \tfrac{4}{3}\pi, \quad
\beta= \tfrac{2}{3}\pi.
\end{align*}
Substituting the above into \eqref{Eq-tan-al-be-ga}, we get
\begin{align*}
\alpha \approx 0.53584\pi, \quad
\beta= \tfrac{2}{3}\pi, \quad
\gamma \approx 0.39874\pi, \quad
x \approx 0.20590\pi.
\end{align*}
These angle values determine all the vertices. Hence we get
\begin{align*}
\AVC = \{ \beta^3, \alpha\beta\gamma^2 \}.
\end{align*}
With the above, one can obtain the last tiling in Figure \ref{Fig-a4-platonic-archimedean-tilings}. 
\end{proof}

\begin{prop}\label{Prop-albe2} The dihedral tilings with vertex $\alpha\beta^2$ are the two tilings in Figure \ref{Fig-a4-Tilings-albe2-albega2}.
\end{prop}

Each of these tilings has $6$ squares and $16$ rhombi. They are the tilings in the second and the third tilings in Figure \ref{Fig-a4-platonic-archimedean-tilings}, first tiling (respectively the second) of which corresponds to the first (respectively the second) of Figure \ref{Fig-a4-Tilings-albe2-albega2}.

\begin{proof} By $\beta>\alpha>\gamma$ and $\alpha\beta^2$ and $\beta+\gamma>\pi$, we get $\pi > \beta > \frac{2}{3}\pi>\alpha > \frac{1}{2}\pi$ and $2\gamma>\alpha$. Then $\alpha > \frac{1}{2}\pi$ and $\alpha\beta^2$ imply $\frac{3}{4}\pi > \beta$. By $\alpha > \frac{1}{2}\pi$ and $\beta+\gamma>\pi$, we also get $\gamma>\frac{1}{2}\alpha>\frac{1}{4}\pi$. By $3\gamma>\frac{3}{4}\pi > \beta$ and $\alpha\beta^2$, we also have $\alpha+\beta+3\gamma>2\pi$. This further implies $\beta^2\cdots=\alpha\beta^2$ and $\alpha\beta\gamma\cdots=\alpha\beta\gamma^2$. We summarise the inequalities below,
\begin{align*}
\alpha > \tfrac{1}{2}\pi, \quad
\beta > \tfrac{2}{3}\pi, \quad
\gamma > \tfrac{1}{4}\pi.
\end{align*}

Lemma \ref{Lem-albe-alga} implies that $\alpha\gamma\cdots$ is a vertex. Then the above inequalities and $\alpha\beta\gamma\cdots=\alpha\beta\gamma^2$ imply
\begin{align}\label{Eq-be3-alga}
\alpha\gamma\cdots=\alpha^3\gamma, \alpha^2\gamma^2, \alpha^2\gamma^3, \alpha\gamma^3, \alpha\gamma^4, \alpha\gamma^5, \alpha\beta\gamma^2.
\end{align}
Similar to the calculation in Proposition \ref{Prop-be3}, any vertex in \eqref{Eq-be3-alga} combined with $\alpha\beta^2$ uniquely determine the angle values through \eqref{Eq-tan-al-be-ga}, which further imply no other vertices with the exception of $\{ \alpha\beta\gamma^2, \alpha\gamma^4 \}$. In particular, if $\alpha\beta\gamma^2, \alpha\gamma^4$ are not vertices, then $\alpha\beta^2$ and exactly one of the remaining vertices in \eqref{Eq-be3-alga} are the only vertices.

We further divide the discussion into the cases below. 

\begin{case*}[One of $\alpha\gamma^3, \alpha\gamma^4, \alpha\gamma^5$] We already know that the vertices are
\begin{align*}
\AVC &= \{ \alpha\beta^2, \alpha\gamma^3 \}; \\
\AVC &= \{ \alpha\beta^2, \alpha\gamma^4, \alpha\beta\gamma^2 \}; \\
\AVC &= \{ \alpha\beta^2, \alpha\gamma^5 \}.
\end{align*}


The arrangement of $\gamma\vert\gamma\vert\gamma$ determines tiles $T_1, T_2, T_3$ in the first picture of Figure \ref{Fig-a4-AAD-albe2-algac/al2gac-gagaga-algaal-alal}. By $\beta_1\beta_2\cdots=\beta_2\beta_3\cdots=\alpha\beta^2$, we also determine $T_4, T_5$. Then we get $\alpha_4\alpha_5\gamma_2\cdots$, a contradiction.

\begin{figure}[h!] 
\centering
\begin{tikzpicture}

\tikzmath{
\s=1;
}

\begin{scope}[]

\tikzmath{
\r=1.2;
\gon=4;
\th=360/\gon;
\x=\r*cos(\th/2);
}

\foreach \a in {0,...,3} {
\draw[rotate=\th*\a] 
	(0:0) -- (0.5*\th:\r)
; 
}

\foreach \a in {0,...,2} {
\draw[rotate=\th*\a] 
	(0.5*\th:\r) -- (0:2*\x)
	(-0.5*\th:\r) -- (0:2*\x)
;
}

\draw[]
	(90:2*\x) -- ([shift=(90:2*\x)]45:0.5*\r)
	(90:2*\x) -- ([shift=(90:2*\x)]135:0.5*\r)
;

\node at (180:0.35*\x) {\small $\gamma$}; 
\node at (180:1.7*\x) {\small $\gamma$};
\node at (1.625*\th:0.85*\r) {\small $\beta$};
\node at (2.375*\th:0.85*\r) {\small $\beta$};

\node at (90:0.35*\x) {\small $\gamma$}; 
\node at (90:1.7*\x) {\small $\gamma$}; 
\node at (0.625*\th:0.85*\r) {\small $\beta$}; 
\node at (1.375*\th:0.85*\r) {\small $\beta$};

\node at (0:0.35*\x) {\small $\gamma$}; 
\node at (0:1.7*\x) {\small $\gamma$};
\node at (0.375*\th:0.85*\r) {\small $\beta$};
\node at (-0.375*\th:0.85*\r) {\small $\beta$};

\node at (100:2*\x) {\small $\alpha$}; 
\node at (1.5*\th:1.15*\r) {\small $\alpha$};

\node at (80:2*\x) {\small $\alpha$}; 
\node at (0.5*\th:1.15*\r) {\small $\alpha$};

\node[inner sep=1,draw,shape=circle] at (180:1*\x) {\small $1$};
\node[inner sep=1,draw,shape=circle] at (90:1*\x) {\small $2$};
\node[inner sep=1,draw,shape=circle] at (0:1*\x) {\small $3$};
\node[inner sep=1,draw,shape=circle] at (1.5*\th:1.5*\r) {\small $4$};
\node[inner sep=1,draw,shape=circle] at (0.5*\th:1.5*\r) {\small $5$};

\end{scope}

\begin{scope}[xshift=5*\s cm, yshift=0.5*\s cm]

\tikzmath{
\r=0.8;
\th=360/4;
\x=\r*cos(0.5*\th);
\R = sqrt(\x^2+(3*\x)^2);
\aR = acos(3*\x/\R);
}

\foreach \aa in {-1,1} {

\tikzset{shift={(\aa*\x,0)}}

\foreach \a in {0,...,3} {

\draw[rotate=\th*\a]
	(0.5*\th:\r) -- (1.5*\th:\r)
;

\node at (0.5*\th+\th*\a: 0.65*\r) {\small $\alpha$}; 

}

}

\foreach \aa in {-1,1} {

\tikzset{shift={(\x,\aa*2*\x)}, yscale=\aa}

\foreach \a in {0,...,3} {

\draw[rotate=\th*\a]
	(0.5*\th:\r) -- (1.5*\th:\r)
;

}

\foreach \a in {0,1} {

\node at (1.5*\th+\a*2*\th: 0.65*\r) {\small $\beta$}; 
\node at (0.5*\th+\a*2*\th: 0.65*\r) {\small $\gamma$}; 

}

}

\node at (2.25*\x, \x) {\small $\beta$};
\node at (2.25*\x, -\x) {\small $\beta$};

\node[inner sep=1,draw,shape=circle] at (-\x,0) {\small $1$};
\node[inner sep=1,draw,shape=circle] at (\x,0) {\small $2$};
\node[inner sep=1,draw,shape=circle] at (\x,2*\x) {\small $3$};
\node[inner sep=1,draw,shape=circle] at (\x,-2*\x) {\small $4$};
\node[inner sep=1,draw,shape=circle] at (3*\x,0) {\small $5$};

\end{scope}

\end{tikzpicture}
\caption{The arrangements of $\gamma\vert\gamma\vert\gamma$ and $\alpha \vert \gamma \vert \alpha$ and $\alpha\vert\alpha$}
\label{Fig-a4-AAD-albe2-algac/al2gac-gagaga-algaal-alal}
\end{figure}
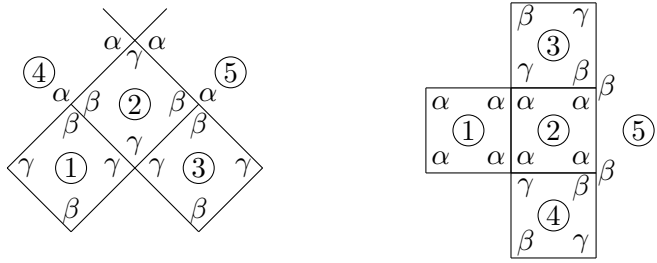

Hence $\alpha\gamma^3, \alpha\gamma^4, \alpha\gamma^5$ are not vertices.
\end{case*}

\begin{case*}[One of $\alpha^2\gamma^2, \alpha^2\gamma^3$] We already know that the vertices are
\begin{align*}
\AVC &= \{ \alpha\beta^2, \alpha^2\gamma^2 \}; \\
\AVC &= \{ \alpha\beta^2, \alpha^2\gamma^3 \}.
\end{align*}
From the above, $\gamma^3, \gamma^4\cdots$ are not vertices and $\alpha\beta\cdots=\alpha\beta^2$ and $\alpha^2\cdots=\alpha\gamma\cdots=\alpha^2\gamma^c$ where $c=2,3$.

The arrangement of $\alpha\vert\alpha$ determines tiles $T_1, T_2$ in the second picture of Figure \ref{Fig-a4-AAD-albe2-algac/al2gac-gagaga-algaal-alal}. Up to mirror symmetry, $\alpha^2\cdots=\alpha^2\gamma^c$ determines $T_3, T_4$. Then $\alpha_2\beta_3\cdots, \alpha_2\beta_4\cdots=\alpha\beta^2$ give two adjacent $\beta$'s in $T_5$, a contradiction. So $\alpha\vert\alpha\cdots$ is not a vertex.

By no $\alpha\vert\alpha\cdots$, we know that $\alpha^2\gamma^c$ with $c=2,3$ has $\alpha\vert \gamma \vert\alpha$. We reverse the deduction in the first picture of Figure \ref{Fig-a4-AAD-albe2-algac/al2gac-gagaga-algaal-alal}. By $\alpha\vert \gamma \vert\alpha$, we determine tiles $T_4, T_2, T_5$. Then $\alpha_4\beta_2\cdots, \alpha_5\beta_2\cdots=\alpha\beta^2$ determine the angles in $T_1, T_3$. This implies $\gamma_1\vert\gamma_2\vert\gamma_3\cdots=\alpha^2\gamma^3$ with $\alpha\vert\alpha$, a contradiction.

Hence $\alpha^2\gamma^2, \alpha^2\gamma^3$ are not vertices.
\end{case*}

\begin{case*}[$\alpha^3\gamma$] We already know that the vertices are
\begin{align*}
\AVC = \{ \alpha\beta^2, \alpha^3\gamma \}.
\end{align*}
Hence $\gamma^3\cdots$ is not a vertex and $\alpha\beta\cdots=\alpha\beta^2$.

For $\alpha\vert\gamma\vert\alpha$, the same reverse deduction in the first picture of Figure \ref{Fig-a4-AAD-albe2-algac/al2gac-gagaga-algaal-alal} implies $\gamma_2\cdots=\gamma^3\cdots$, a contradiction. Hence $\alpha^3\gamma$ is not a vertex.
\end{case*}

\begin{case*}[$\alpha\beta\gamma^2$] We have already shown that $\alpha\gamma^4$ is not a vertex. So the vertices are 
\begin{align*}
\AVC = \{ \alpha\beta^2, \alpha\beta\gamma^2 \}.
\end{align*}
Then by \eqref{Eq-tan-al-be-ga}, we get
\begin{align*}
\alpha \approx 0.55138\pi, \quad
\beta \approx 0.72431\pi, \quad
\gamma \approx 0.36216\pi, \quad
x \approx 0.24273\pi.
\end{align*}
We know that $\alpha^2\cdots, \beta^2\gamma\cdots$ are not vertices and $\alpha\cdots=\alpha\beta^2, \alpha\beta\gamma^2$ and $\beta^2\cdots=\alpha\beta^2$ and $\gamma\cdots=\alpha\beta\gamma^2$.

We know that $\alpha\beta^2 = \vert \alpha \vert \beta \vert \beta \vert$ is a vertex. By the second picture of Figure \ref{Fig-a4-adj-square-rhombus}, $\beta \vert \beta \cdots$ is a vertex if and only if $\gamma\vert\gamma \cdots$ is a vertex. Then $\gamma \vert \gamma \cdots = \vert \alpha \vert \beta \vert \gamma \vert \gamma \vert$ a vertex. We denote such vertex by \quotes{$\bullet$}. Hence there is a square with at least one $\bullet$. We classify the distribution of $\bullet$'s.


Suppose a square has diagonal $\bullet$'s arranged as in the first picture of Figure \ref{Fig-a4-albe2-sq-diagonal-albegaga}. Then we have tiles $T_1, T_2, T_3, T_4, T_5$. By $\beta_2\beta_4\cdots,\beta_3\beta_5\cdots=\alpha\beta^2$, we further determine $T_6, T_7$, contradicting $\gamma_4\gamma_5\cdots=\alpha\beta\gamma^2$. Then a square with diagonal $\bullet$'s must be arranged as in the second picture. This implies that the other two vertices at the square are also $\bullet$'s (indicated by \textcolor{gray!75}{$\bullet$}'s).

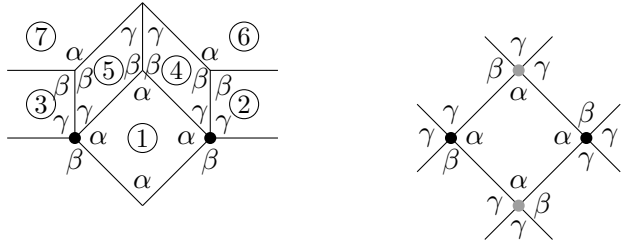
\begin{figure}[h!]
\centering
\begin{tikzpicture}[scale=1]

\tikzmath{
\s = 1;
\r= 0.9;
\th=360/4;
\x=\r*cos(0.5*\th);
}

\begin{scope}[]

\foreach \a in {0,1,2,3} {

\draw[rotate=\a*\th]
	(90:\r) -- (90+\th:\r)
;

\node at (90+\a*\th:0.65*\r) {\small $\alpha$};

}

\foreach \a in {0,1,2} {

\draw[rotate=\a*\th]
	(0:\r) -- (0:2*\r)
;

}

\foreach \aa in {-1,1} {

\tikzset{xscale=\aa}

\draw[] 
	(0:\r) -- (\r,\r)
	(\r,\r) -- (90:2*\r) 
	 (\r,\r) --  (2*\r,\r)
;

\node at (1.2*\r, 0.2*\r) {\small $\gamma$}; 
\node at (1.2*\r, 0.75*\r) {\small $\beta$}; 

\node at (1*\r, -0.35*\r) {\small $\beta$};

\node at (0.85*\r, 0.35*\r) {\small $\gamma$}; 
\node at (0.15*\r, 1.05*\r) {\small $\beta$};
\node at (0.2*\r, 1.5*\r) {\small $\gamma$};
\node at (0.85*\r, 0.85*\r) {\small $\beta$};

\node at (1*\r, 1.25*\r) {\small $\alpha$}; 

}

	(90:2*\r)  -- (90:2.5*\r) 
;

\node [shape = circle, fill = black, minimum size = 0.16 cm, inner sep=0pt] at (0:\r) {};
\node [shape = circle, fill = black, minimum size = 0.16 cm, inner sep=0pt] at (180:\r) {};

\node[inner sep=1,draw,shape=circle] at (0,0) {\footnotesize $1$};
\node[inner sep=1,draw,shape=circle] at (1.5*\r,0.5*\r) {\footnotesize $2$};
\node[inner sep=1,draw,shape=circle] at (-1.5*\r,0.5*\r) {\footnotesize $3$};
\node[inner sep=1,draw,shape=circle] at (0.5*\r,\r) {\footnotesize $4$};
\node[inner sep=1,draw,shape=circle] at (-0.5*\r,\r) {\footnotesize $5$};

\node[inner sep=1,draw,shape=circle] at (1.5*\r,1.5*\r) {\footnotesize $6$};
\node[inner sep=1,draw,shape=circle] at (-1.5*\r,1.5*\r) {\footnotesize $7$};

\end{scope}

\begin{scope}[xshift=5*\s cm]

\foreach \a in {0,1,2,3} {

\draw[rotate=\a*\th]
	(90:\r) -- (90+\th:\r)
	(0:\r) -- (1.5*\r,0.5*\r)
	(0:\r) -- (1.5*\r,-0.5*\r)
;

\node at (90+\a*\th:0.65*\r) {\small $\alpha$};

\node at ([rotate=\a*\th]1.35*\r, 0) {\small $\gamma$}; 
\node at ([rotate=\a*\th]1*\r, 0.35*\r) {\small $\beta$}; 
\node at ([rotate=\a*\th]1*\r, -0.35*\r) {\small $\gamma$};

}

\foreach \a in {0,2} {

\node [shape = circle, fill = black, minimum size = 0.16 cm, inner sep=0pt] at (0+\a*\th:\r) {};
\node [shape = circle, fill = gray!75, minimum size = 0.16 cm, inner sep=0pt] at (\th+\a*\th:\r) {};

}

\end{scope}

\end{tikzpicture}
\caption{Diagonal $\vert \alpha \vert \beta \vert \gamma \vert \gamma \vert$'s at a square}
\label{Fig-a4-albe2-sq-diagonal-albegaga}
\end{figure}

Therefore there are three possible distributions of $\bullet$'s in Figure \ref{Fig-a4-albe2-sq-albegaga-distribution}. The first picture is the second of Figure \ref{Fig-a4-albe2-sq-diagonal-albegaga}. The second and the third picture of Figure \ref{Fig-a4-albe2-sq-albegaga-distribution} have two adjacent $\bullet$'s and exactly one $\bullet$ respectively. By similar argument, we uniquely determine the angle arrangements of $T_1, T_2, ..., T_5$ in these pictures. The third picture leads to a contradiction.


\begin{figure}[h!]
\centering
\begin{tikzpicture}[scale=1]

\tikzmath{
\s = 1;
\r= 0.8;
\th=360/4;
\x=\r*cos(0.5*\th);
}

\begin{scope}[] 

\foreach \a in {0,1,2,3} {

\draw[rotate=\a*\th]
	(0.5*\th:\r) -- (1.5*\th:\r)
	(0.5*\th:\r) -- (\x,\x+0.5*\r)
	(0.5*\th:\r) -- (\x+0.5*\r,\x)
;

\node [shape = circle, fill = black, minimum size = 0.16 cm, inner sep=0pt] at (0.5*\th+\a*\th:\r) {};

\node at (0.5*\th+\a*\th:0.65*\r) {\small $\alpha$};

\node at ([rotate=\a*\th]0.65*\x,1.35*\x) {\small $\beta$};
\node at ([rotate=\a*\th]1.35*\x,0.65*\x) {\small $\gamma$};
\node at ([rotate=\a*\th]1.35*\x,1.35*\x) {\small $\gamma$};

}


\node[inner sep=1,draw,shape=circle] at (0,0) {\small $1$};
\node[inner sep=1,draw,shape=circle] at (0,2*\x) {\small $2$};
\node[inner sep=1,draw,shape=circle] at (2*\x,0) {\small $3$};
\node[inner sep=1,draw,shape=circle] at (0,-2*\x) {\small $4$};
\node[inner sep=1,draw,shape=circle] at (-2*\x,0) {\small $5$};

\end{scope}

\begin{scope}[xshift=3.25*\s cm] 

\foreach \a in {0,1,2,3} {

\draw[rotate=\a*\th]
	(0.5*\th:\r) -- (1.5*\th:\r)
;

\node at (0.5*\th+\a*\th:0.65*\r) {\small $\alpha$};

}

\foreach \a in {0,1,2} {

\draw[rotate=\a*\th]
	(0.5*\th:\r) -- (\x,\x+0.5*\r)
	(0.5*\th:\r) -- (\x+0.5*\r,\x)
;

}

\draw[]
	(-0.5*\th:\r) -- (-0.5*\th:1.5*\r)
;

\foreach \a in {0,1} {

\node at ([rotate=\a*\th]0.65*\x,1.35*\x) {\small $\beta$};
\node at ([rotate=\a*\th]1.35*\x,0.65*\x) {\small $\gamma$};
\node at ([rotate=\a*\th]1.35*\x,1.35*\x) {\small $\gamma$};

\node [shape = circle, fill = black, minimum size = 0.16 cm, inner sep=0pt] at (0.5*\th+\a*\th:\r) {};

}

\node at (0.75*\x,-1.35*\x) {\small $\beta$};
\node at (1.35*\x,-0.75*\x) {\small $\beta$};

\node at (-1.35*\x,-1.35*\x) {\small $\beta$};
\node at (-0.7*\x,-1.35*\x) {\small $\gamma$};
\node at (-1.3*\x,-0.65*\x) {\small $\gamma$};

\node[inner sep=1,draw,shape=circle] at (0,0) {\small $1$};
\node[inner sep=1,draw,shape=circle] at (0,2*\x) {\small $2$};
\node[inner sep=1,draw,shape=circle] at (2*\x,0) {\small $3$};
\node[inner sep=1,draw,shape=circle] at (0,-2*\x) {\small $4$};
\node[inner sep=1,draw,shape=circle] at (-2*\x,0) {\small $5$};

\end{scope}

\begin{scope}[xshift=6.5*\s cm] 

\foreach \a in {0,1,2,3} {

\draw[rotate=\a*\th]
	(0.5*\th:\r) -- (1.5*\th:\r)
;

\node at (0.5*\th+\a*\th:0.65*\r) {\small $\alpha$};

}

\foreach \a in {0,1,2} {

\draw[rotate=\a*\th]
	(0.5*\th:\r) -- (\x,\x+0.5*\r)
	(0.5*\th:\r) -- (\x+0.5*\r,\x)
;

}

\draw[]
	(-0.5*\th:\r) -- (-0.5*\th:1.5*\r)
;

\draw[]
	(-0.5*\th:\r) -- (-0.5*\th:1.5*\r)
;

\foreach \a in {0} {

\node at ([rotate=\a*\th]0.65*\x,1.35*\x) {\small $\beta$};
\node at ([rotate=\a*\th]1.35*\x,0.65*\x) {\small $\gamma$};
\node at ([rotate=\a*\th]1.35*\x,1.35*\x) {\small $\gamma$};

}

\node [shape = circle, fill = black, minimum size = 0.16 cm, inner sep=0pt] at (0.5*\th:\r) {};

\node at (-1.35*\x,1.35*\x) {\small $\beta$};
\node at (-0.7*\x,1.35*\x) {\small $\gamma$};
\node at (-1.3*\x,0.65*\x) {\small $\gamma$};

\node at (0.75*\x,-1.35*\x) {\small $\beta$};
\node at (1.35*\x,-0.75*\x) {\small $\beta$};

\node at (-1.35*\x,-1.35*\x) {\small $?$};
\node at (-0.7*\x,-1.35*\x) {\small $\gamma$};
\node at (-1.3*\x,-0.65*\x) {\small $\beta$};

\node[inner sep=1,draw,shape=circle] at (0,0) {\small $1$};
\node[inner sep=1,draw,shape=circle] at (0,2*\x) {\small $2$};
\node[inner sep=1,draw,shape=circle] at (2*\x,0) {\small $3$};
\node[inner sep=1,draw,shape=circle] at (0,-2*\x) {\small $4$};
\node[inner sep=1,draw,shape=circle] at (-2*\x,0) {\small $5$};

\end{scope}

\end{tikzpicture}
\caption{Distributions of $\vert \alpha \vert \beta \vert \gamma \vert \gamma \vert$'s at a square}
\label{Fig-a4-albe2-sq-albegaga-distribution}
\end{figure}
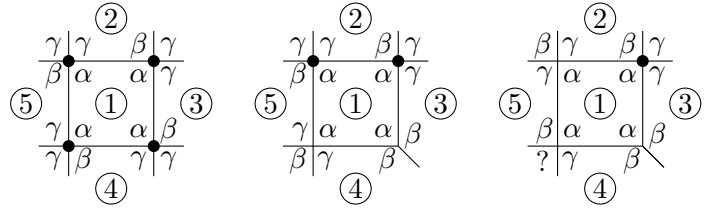



Now we know that a square in the tiling has the angle arrangement of the first or the second picture of Figure \ref{Fig-a4-albe2-sq-albegaga-distribution}. The first angle arrangement (respectively the second) of Figure \ref{Fig-a4-albe2-sq-albegaga-distribution} uniquely determines the first (respectively the second) tiling of Figure \ref{Fig-a4-Tilings-albe2-albega2} where the centre square is $T_1$. The second tiling is drawn with a different perspective from the same tiling (second one) in the third tiling of Figure \ref{Fig-a4-platonic-archimedean-tilings}.


\begin{figure}[h!]
\centering
\begin{tikzpicture}[scale=1]

\tikzmath{
\s = 1;
\r= 0.4;
\th=360/4;
\ph=360/8;
\x=\r*cos(0.5*\th);
}

\begin{scope}[] 

\fill[gray!50] (0,0) circle (4.5*\r);
	
\fill[white] (0,0) circle (4*\r);

\fill[gray!50]
	(0.5*\th:\r) -- (1.5*\th:\r) -- (2.5*\th:\r) -- (3.5*\th:\r) -- cycle
;

\foreach \a in {0,...,3} {

\fill[gray!50, rotate=\a*\th]
	(90:2*\r) -- (90:3*\r) -- (90-0.5*\th:3*\r) -- (90-0.5*\th:2*\r) -- cycle
;

}

\foreach \a in {0,...,3} {

\draw[rotate=\a*\th]
	(0.5*\th:\r) -- (1.5*\th:\r)
	(1.5*\th:\r) -- (90:2*\r)
	(0.5*\th:\r) -- (90-\ph:2*\r)
	(90:3*\r) -- (90:4*\r)
	(90+0.5*\th:3*\r) -- (90:4*\r)
;

}

\foreach \a in {1,3,...,7} {

\draw[rotate=\a*\ph]
	(90:2*\r) -- (90+\ph:2*\r)
	(90:3*\r) -- (90+\ph:3*\r)
;

}

\foreach \a in {0,2,...,6} {

\draw[rotate=\a*\ph]
	(90+0.5*\th:2*\r) -- (90:3*\r)
;

	(90:2*\r) -- (90-0.5*\th:\r)
	(90:2*\r) -- (90+\ph:2*\r)
	(90:3*\r) -- (90+\ph:3*\r)
	(90-0.5*\th:3*\r) -- (90:4*\r)
;

}

\foreach \a in {0,...,7} {

\draw[rotate=\a*\ph]
	(90:2*\r) -- (90:3*\r)
;

}

\draw[] (0,0) circle (4*\r);

\foreach \a in {0,...,3} {

\node [shape = circle, fill = black, minimum size = 0.16 cm, inner sep=0pt] at (0.5*\th+\a*\th:\r) {};
\node [shape = circle, fill = black, minimum size = 0.16 cm, inner sep=0pt] at (90+\a*\th: 4*\r) {};

}

\end{scope} 

\begin{scope}[xshift=4.35*\s cm] 

\fill[gray!50] (0,0) circle (4.5*\r);
	
\fill[white] (0,0) circle (4*\r);

\fill[gray!50]
	(0.5*\th:\r) -- (1.5*\th:\r) -- (2.5*\th:\r) -- (3.5*\th:\r) -- cycle
;

\foreach \a in {0,...,3} {

\fill[gray!50, rotate=\a*\th]
	(90:2*\r) -- (90:3*\r) -- (90-0.5*\th:3*\r) -- (90-0.5*\th:2*\r) -- cycle
;

}

\foreach \a in {0,...,3} {

\draw[rotate=\a*\th]
	(0.5*\th:\r) -- (1.5*\th:\r)
	(90:2*\r) -- (0.5*\th:2*\r)
	(90:3*\r) -- (0.5*\th:3*\r)
;

}

\foreach \a in {0,...,7} {

\draw[rotate=\a*\ph]
	(90:2*\r) -- (90:3*\r)
;

}

\foreach \a in {0,1,2} {

\draw[rotate=\a*\th]
	(-0.5*\th:\r) -- (-0.5*\th:2*\r)
	(0:3*\r) -- (0:4*\r)
;

}

\foreach \a in {0,1} {

\draw[rotate=\a*180]
	(2.5*\th:\r) -- (180:2*\r)
;

\draw[rotate=\a*\th]
	(180:2*\r) -- (2.5*\th:2*\r) 
	(0.5*\th:2*\r) -- (0:3*\r)
	(0.5*\th:3*\r) -- (0:4*\r)
	(180:3*\r) -- (2.5*\th:3*\r)
;

}

\draw[]
	(1.5*\th:\r) -- (90:2*\r)
	(2.5*\th:\r) -- (270:2*\r)
	(2.5*\th:3*\r) -- (270:4*\r)
	(-0.5*\th:3*\r) -- (270:4*\r)
;

\draw[] (0,0) circle (4*\r);

\node [shape = circle, fill = black, minimum size = 0.16 cm, inner sep=0pt] at (0.5*\th:\r) {};
\node [shape = circle, fill = black, minimum size = 0.16 cm, inner sep=0pt] at (1.5*\th:\r) {};

\node [shape = circle, fill = black, minimum size = 0.16 cm, inner sep=0pt] at (-0.5*\th:2*\r) {};
\node [shape = circle, fill = black, minimum size = 0.16 cm, inner sep=0pt] at (-0.5*\th:3*\r) {};

\node [shape = circle, fill = black, minimum size = 0.16 cm, inner sep=0pt] at (180:2*\r) {};
\node [shape = circle, fill = black, minimum size = 0.16 cm, inner sep=0pt] at (180:3*\r) {};

\node [shape = circle, fill = black, minimum size = 0.16 cm, inner sep=0pt] at (0:4*\r) {};
\node [shape = circle, fill = black, minimum size = 0.16 cm, inner sep=0pt] at (90:4*\r) {};

\end{scope}

\end{tikzpicture}
\caption{The two tilings with $\alpha\beta^2, \alpha\beta\gamma^2$ and $6$ squares and $16$ rhombi} 
\label{Fig-a4-Tilings-albe2-albega2}
\end{figure}
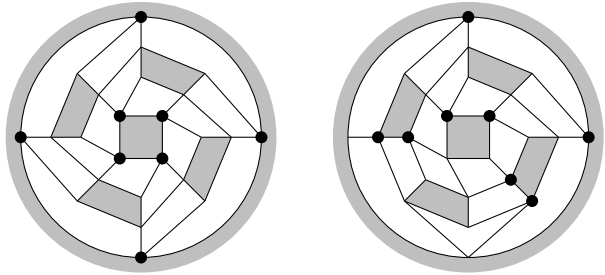

By $\beta=2\gamma$, the rhombus (the first picture of Figure \ref{Fig-a4-albe2-albega2-reg-tri}) in the tiling can be subdivided into two regular triangles along the diagonal between two $\beta$'s. We get a dihedral tiling by regular triangles and squares with $\alpha\gamma^4$ as the only vertex. This tiling is the snub cube in the second picture. Therefore the tilings in Figure \ref{Fig-a4-Tilings-albe2-albega2} are the triangular fusion of the snub cube.

\begin{figure}[h!] 
\centering
\begin{tikzpicture}

\tikzmath{
\s=1;
}

\begin{scope}[yshift=0.35 cm]

\tikzmath{
\r=0.8;
\th=360/3;
\x = \r*cos(0.5*\th);
}

\foreach \a in {0,2} {

\draw[rotate=\a*\th]
	(90:\r) -- (90+\th:\r)
;

\tikzset{shift={(0,-2*\x)}}

\draw[rotate=\a*\th]
	(270:\r) -- (270+\th:\r)
;

}

\draw[gray, dashed]
	(90-\th:\r) --  (90+\th:\r)
;

\node at (-0.5*\r, -1*\x) {\small $\beta$};
\node at (0.5*\r, -1*\x) {\small $\beta$};
\node at (0,0.5*\r) {\small $\gamma$};
\node at (0,-2*\x-0.5*\r) {\small $\gamma$};

\end{scope}

\begin{scope}[xshift=3.5*\s cm] 

\tikzmath{
\r= 0.375;
\th=360/4;
\ph=360/8;
\x=\r*cos(0.5*\th);
}

\fill[gray!50] (0,0) circle (4.35*\r);
	
\fill[white] (0,0) circle (4*\r);

\fill[gray!50]
	(0.5*\th:\r) -- (1.5*\th:\r) -- (2.5*\th:\r) -- (3.5*\th:\r) -- cycle
;

\foreach \a in {0,...,3} {

\fill[gray!50, rotate=\a*\th]
	(90:2*\r) -- (90:3*\r) -- (90-0.5*\th:3*\r) -- (90-0.5*\th:2*\r) -- cycle
;

}

\foreach \a in {0,...,3} {

\draw[rotate=\a*\th]
	(0.5*\th:\r) -- (1.5*\th:\r)
	(1.5*\th:\r) -- (90:2*\r)
	(0.5*\th:\r) -- (90-\ph:2*\r)
	(90:3*\r) -- (90:4*\r)
	(90+0.5*\th:3*\r) -- (90:4*\r)
;


}

\foreach \a in {1,3,...,7} {

\draw[rotate=\a*\ph]
	(90:2*\r) -- (90+\ph:2*\r)
	(90:3*\r) -- (90+\ph:3*\r)
;

}

\foreach \a in {0,2,...,6} {

\draw[rotate=\a*\ph]
	(90+0.5*\th:2*\r) -- (90:3*\r)
;

\draw[
rotate=\a*\ph]
	(90:2*\r) -- (90-0.5*\th:\r)
	(90:2*\r) -- (90+\ph:2*\r)
	(90:3*\r) -- (90+\ph:3*\r)
	(90-0.5*\th:3*\r) -- (90:4*\r)
;

}

\foreach \a in {0,...,7} {

\draw[rotate=\a*\ph]
	(90:2*\r) -- (90:3*\r)
;

}

\draw[] (0,0) circle (4*\r);

\end{scope} 

\end{tikzpicture}
\caption{Rhombus as a union of two regular triangles in the snub cube} \qedhere
\label{Fig-a4-albe2-albega2-reg-tri}
\end{figure}
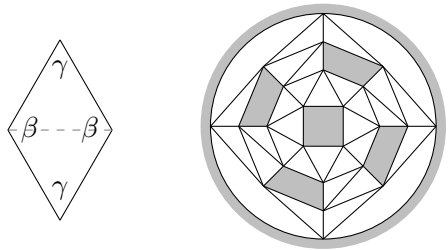

\end{case*}
\end{proof}

\begin{prop} \label{Prop-al2be} The dihedral tilings with vertex $\alpha^2\beta$ are the two tilings in Figure \ref{Fig-a4-Tilings-al2be-al3ga}.
\end{prop}

Each of these tilings has $10$ squares and $4$ rhombi. These are the tilings in Figure \ref{Fig-a4-sporadic-tilings}.

\begin{proof} By $\beta>\alpha$ and $\alpha^2\beta$, we have $\beta>\frac{2}{3}\pi>\alpha$ and $R(\beta^2)<\alpha$. By $\beta+\gamma>\pi$, we have $R(\beta^2)<2\gamma$. Then by $\beta>\alpha>\gamma$, we get $\beta^2\cdots=\beta^2\gamma$.

Suppose both $\alpha^2\beta, \beta^2\gamma$ are vertices. Then we get
\begin{align*}
\beta =2\pi - 2\alpha, \quad
\gamma = 4\alpha- 2\pi.
\end{align*}
Substituting the above into \eqref{Eq-tan-al-be-ga}, we get
\begin{align*}
\tan^2 \tfrac{1}{2}\alpha = - \tan \alpha \tan 2\alpha = - \frac{8\tan^2 \tfrac{1}{2}\alpha }{ \tan^4 \tfrac{1}{2}\alpha - 6 \tan^2 \tfrac{1}{2}\alpha + 1 },
\end{align*}
which gives
\begin{align*}
\tan \tfrac{1}{2}\alpha = 0, \pm \sqrt{3}.
\end{align*}
Hence $\alpha = 0, \tfrac{2}{3}\pi, \tfrac{4}{3}\pi$, which contradict $\beta> \tfrac{2}{3}\pi > \alpha > \frac{1}{2}\pi$. Therefore $\beta^2\cdots=\beta^2\gamma$ is not a vertex. Combined with Propositions \ref{Prop-al3}, \ref{Prop-albega}, \ref{Prop-be3}, \ref{Prop-albe2}, we may assume that $\alpha^2\beta$ is the only degree $3$ vertex.

By $\alpha^2\beta$ and $\beta>\alpha$, we get $\beta>\frac{2}{3}\pi$. Then by $\alpha>\frac{1}{2}\pi$ the vertices are 
\begin{align*}
\alpha^2\beta, \gamma^c, \alpha^{a\le3}\gamma^c, \beta\gamma^c, \alpha\beta\gamma^c.
\end{align*}

By no $\beta^2\cdots$, we know that $\beta\vert\beta\cdots$ is not a vertex and therefore $\gamma\vert\gamma\cdots$ is not a vertex (as shown in Figure \ref{Fig-a4-adj-square-rhombus}). This implies that $\gamma^c, \alpha\gamma^c, \beta\gamma^c$ are not vertices, and $c\le3$ in $\alpha^3\gamma^c$, and $c\le2$ in $\alpha^2\gamma^c, \alpha\beta\gamma^c$, and $\alpha\vert\beta\cdots=\alpha^2\beta$. Hence the vertices are
\begin{align*}
\alpha^2\beta, \alpha^{3}\gamma, \alpha^{3}\gamma^2, \alpha^{3}\gamma^3, \alpha^{2}\gamma^2, \alpha\beta\gamma^2.
\end{align*}

Assume one of $\alpha^2\gamma^2, \alpha\beta\gamma^2, \alpha^3\gamma^2, \alpha^3\gamma^3$ is a vertex. By no $\gamma\vert\gamma\cdots$, there is $\gamma\vert\alpha\vert\gamma$ at the vertex, which determines tiles $T_1, T_2, T_3$ in Figure \ref{Fig-a4-AAD-al2be-gaalga}. By $\alpha_2\vert\beta_1\cdots, \alpha_2\vert\beta_3\cdots=\alpha^2\beta$, we further get $T_4, T_5$. This gives $\alpha_4 \vert \alpha_2 \vert \alpha_5 \cdots$. Combined with no $\gamma\vert\gamma\cdots$, we know that $\alpha^3\gamma$ is a vertex.

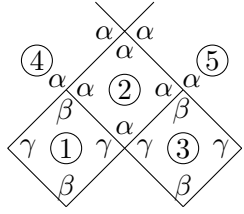
\begin{figure}[h!] 
\centering
\begin{tikzpicture}

\tikzmath{
\s=1;
\r=1.1;
\gon=4;
\th=360/\gon;
\x=\r*cos(\th/2);
}

\begin{scope}

\foreach \a in {0,...,3} {
\draw[rotate=\th*\a] 
	(0:0) -- (0.5*\th:\r)
; 
}

\foreach \a in {0,...,2} {
\draw[rotate=\th*\a] 
	(0.5*\th:\r) -- (0:2*\x)
	(-0.5*\th:\r) -- (0:2*\x)
;
}

\draw[]
	(90:2*\x) -- ([shift=(90:2*\x)]45:0.5*\r)
	(90:2*\x) -- ([shift=(90:2*\x)]135:0.5*\r)
;

\node at (180:0.35*\x) {\small $\gamma$}; 
\node at (180:1.65*\x) {\small $\gamma$};
\node at (1.65*\th:0.85*\r) {\small $\beta$};
\node at (2.375*\th:0.85*\r) {\small $\beta$};

\node at (90:0.35*\x) {\small $\alpha$}; 
\node at (90:1.65*\x) {\small $\alpha$}; 
\node at (0.625*\th:0.85*\r) {\small $\alpha$}; 
\node at (1.375*\th:0.85*\r) {\small $\alpha$};

\node at (0:0.35*\x) {\small $\gamma$}; 
\node at (0:1.65*\x) {\small $\gamma$};
\node at (0.375*\th:0.825*\r) {\small $\beta$};
\node at (-0.375*\th:0.85*\r) {\small $\beta$};

\node at (100:2*\x) {\small $\alpha$}; 
\node at (1.5*\th:1.15*\r) {\small $\alpha$};

\node at (80:2*\x) {\small $\alpha$}; 
\node at (0.5*\th:1.15*\r) {\small $\alpha$};

\node[inner sep=1,draw,shape=circle] at (180:1*\x) {\small $1$};
\node[inner sep=1,draw,shape=circle] at (90:1*\x) {\small $2$};
\node[inner sep=1,draw,shape=circle] at (0:1*\x) {\small $3$};
\node[inner sep=1,draw,shape=circle] at (1.5*\th:1.5*\r) {\small $4$};
\node[inner sep=1,draw,shape=circle] at (0.5*\th:1.5*\r) {\small $5$};

\end{scope}

\end{tikzpicture}
\caption{The arrangement of $\gamma\vert\alpha\vert\gamma$}
\label{Fig-a4-AAD-al2be-gaalga}
\end{figure}

By $\alpha^2\beta, \alpha^3\gamma$ and \eqref{Eq-tan-al-be-ga}, we have
\begin{align*}
\alpha \approx 0.58043\pi, \quad
\beta \approx 0.83914\pi, \quad
\gamma \approx 0.25871\pi, \quad
x \approx 0.29517\pi.
\end{align*}
This implies that $\alpha^2\gamma^2, \alpha\beta\gamma^2, \alpha^3\gamma^2, \alpha^3\gamma^3$ are not verices. Hence we get
\begin{align}\label{Eq-AVC-al2be-al3ga}
\AVC = \{ \alpha^2\beta, \alpha^3\gamma \}.
\end{align}
From the above, we know $\beta\cdots=\alpha^2\beta$ and $\gamma\cdots=\alpha^3\gamma$.

The arrangement of $\alpha^3\gamma$ determines tiles $T_1, T_2, T_3$ in Figure \ref{Fig-a4-AAD-al2be-V-form}. We know $\alpha_1\alpha_2\cdots=\alpha^2\beta, \alpha^3\gamma$. If the undetermined $\alpha_1\alpha_2\cdots=\alpha^2\beta$, then we determine $T_5$ in the first picture. Then $\alpha_2\gamma_5\cdots=\alpha^3\gamma$ determines $T_4$. If the undetermined $\alpha_1\alpha_2\cdots=\alpha^3\gamma$ as configured in the second picture, then we determine $T_5$. Then $\alpha_2\beta_5\cdots=\alpha^2\beta$ determines $T_4$. We further get the undetermined $\alpha_2\alpha_3\cdots=\alpha^3\gamma$. If the undetermined $\alpha_1\alpha_2\cdots=\alpha^3\gamma$ is not as configured in the second picture, then after a horizontal flip, the situation is described by the $\alpha_2\alpha_3\cdots=\alpha^3\gamma$ as arranged. In each situation, we know that there are four squares in the configuration of $T_1, T_2, T_3, T_4$.

\begin{figure}[h!] 
\centering
\begin{tikzpicture}

\tikzmath{
\s=1;
\r=1.2;
\th=360/4;
\x=\r*cos(0.5*\th);
\R = sqrt(\x^2+(3*\x)^2);
\aR = acos(3*\x/\R);
}

\begin{scope}[] 

\foreach \a in {0,...,3} {

\draw[rotate=\th*\a]
	(0:0) -- (90-0.5*\th:\r)
;

}

\foreach \a in {1,...,3} {

\draw[rotate=\th*\a]
	(90:2*\x) -- (90-0.5*\th:\r)
	(90:2*\x) -- (90+0.5*\th:\r)
;

}

\foreach \aa in {0,2,3} {

\tikzset{shift={(\th*\aa:\x)}}

\foreach \a in {0,...,3} {

\node at (\th+\th*\a: 0.45*\r) {\small $\alpha$}; 

}
}

\foreach \a in {0,...,3} {

\tikzset{shift={(\x, -2*\x)}}

\draw[rotate=\th*\a]
	(\th:\x) -- (\th+\th:\x)
;

}

\draw[]
	(-2*\x, 0) -- (-2*\x, -2*\x)
	(-2*\x, -2*\x) -- (0, -2*\x)
	%
;

	(-\x, \x) -- (-\x, 2*\x) 
;

\node at (90:0.3*\x) {\small $\gamma$};

\node at (-1.15*\x, -1.15*\x) {\small $\beta$}; 
\node at (-1.8*\x, -0.5*\x) {\small $\gamma$};
\node at (-0.5*\x, -1.8*\x) {\small $\gamma$};
\node at (-1.8*\x, -1.8*\x) {\small $\beta$};



\node at (\x, -1.35*\x) {\small $\alpha$}; 
\node at (0.35*\x, -2*\x) {\small $\alpha$};
\node at (1.65*\x, -2*\x) {\small $\alpha$};
\node at (\x, -2.65*\x) {\small $\alpha$};

\node at (-0.1*\x, -2.25*\x) {\small $\alpha$};

\node[inner sep=1,draw,shape=circle] at (180:\x) {\small $1$};
\node[inner sep=1,draw,shape=circle] at (270:\x) {\small $2$};
\node[inner sep=1,draw,shape=circle] at (0:\x) {\small $3$};
\node[inner sep=1,draw,shape=circle] at (-1.475*\x,-1.475*\x) {\small $5$};
\node[inner sep=1,draw,shape=circle] at (\x,-2*\x) {\small $4$};

\end{scope} 

\begin{scope}[xshift=5*\s cm] 

\foreach \a in {0,...,3} {

\draw[rotate=\th*\a]
	(0:0) -- (90-0.5*\th:\r)
;

}

\foreach \a in {1,...,3} {

\draw[rotate=\th*\a]
	(90:2*\x) -- (90-0.5*\th:\r)
	(90:2*\x) -- (90+0.5*\th:\r)
;

}

\foreach \aa in {0,2,3} {

\tikzset{shift={(\th*\aa:\x)}}

\foreach \a in {0,...,3} {

\node at (\th+\th*\a: 0.45*\r) {\small $\alpha$}; 

}
}

\foreach \a in {0,...,3} {

\tikzset{shift={(\x, -2*\x)}}

\draw[rotate=\th*\a]
	(\th:\x) -- (\th+\th:\x)
;

}

\draw[]
	(-\x, -\x) -- (-\x, -3*\x)
	(-\x, -3*\x) -- (\x, -3*\x)
	%
;

\node at (90:0.3*\x) {\small $\gamma$};

\node at (-1.25*\x, -1.1*\x) {\small $\alpha$};

\node at (1.3*\x, -1.05*\x) {\small $\gamma$};

\node at (-0.15*\x, -2.15*\x) {\small $\beta$}; 
\node at (-0.8*\x, -1.5*\x) {\small $\gamma$};
\node at (0.5*\x, -2.8*\x) {\small $\gamma$};
\node at (-0.8*\x, -2.8*\x) {\small $\beta$};

\node at (\x, -1.35*\x) {\small $\alpha$}; 
\node at (0.35*\x, -2*\x) {\small $\alpha$};
\node at (1.65*\x, -2*\x) {\small $\alpha$};
\node at (\x, -2.65*\x) {\small $\alpha$};


\node[inner sep=1,draw,shape=circle] at (180:\x) {\small $1$};
\node[inner sep=1,draw,shape=circle] at (270:\x) {\small $2$};
\node[inner sep=1,draw,shape=circle] at (0:\x) {\small $3$};
\node[inner sep=1,draw,shape=circle] at (-0.475*\x,-2.475*\x) {\small $5$};
\node[inner sep=1,draw,shape=circle] at (\x,-2*\x) {\small $4$};

\end{scope}

\end{tikzpicture}
\caption{The arrangement of $\alpha^3\gamma$}
\label{Fig-a4-AAD-al2be-V-form}
\end{figure}
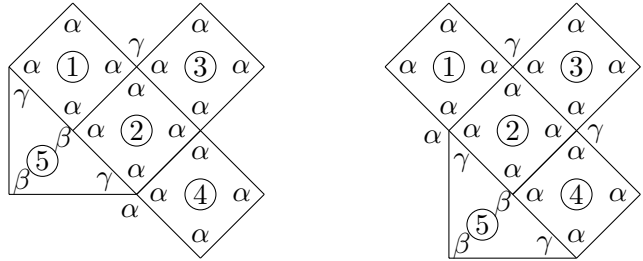

For $T_1, T_2, T_3, T_4$ in Figure \ref{Fig-a4-AAD-al2be-V-form}, there are two possibilities, a square or a rhombus, for the adjacent tile along the remaining edge of $T_2$. If it is a square, then we determine $T_5$ adjacent to $T_2$ in the first picture of Figure \ref{Fig-a4-Tilings-al2be-al3ga}. Then $\alpha_1\alpha_2\alpha_3\cdots, \alpha_1\alpha_2\alpha_5\cdots, \alpha_2\alpha_4\alpha_5\cdots, \alpha_2\alpha_3\alpha_4\cdots =\alpha^3\gamma$ determine $T_6, T_7, T_8, T_9$ respectively. Then $\alpha_1\beta_6\cdots, \alpha_1\beta_7\cdots=\alpha^2\beta$ determine $T_{10}$. By rotational symmetry, we also determine $T_{11}, T_{12}, T_{13}$. By $\alpha_{10}\alpha_{13}\gamma_6\cdots=\alpha^3\gamma$, we determine $T_{14}$ and hence a tiling with $10$ squares and $4$ rhombi in the first picture.


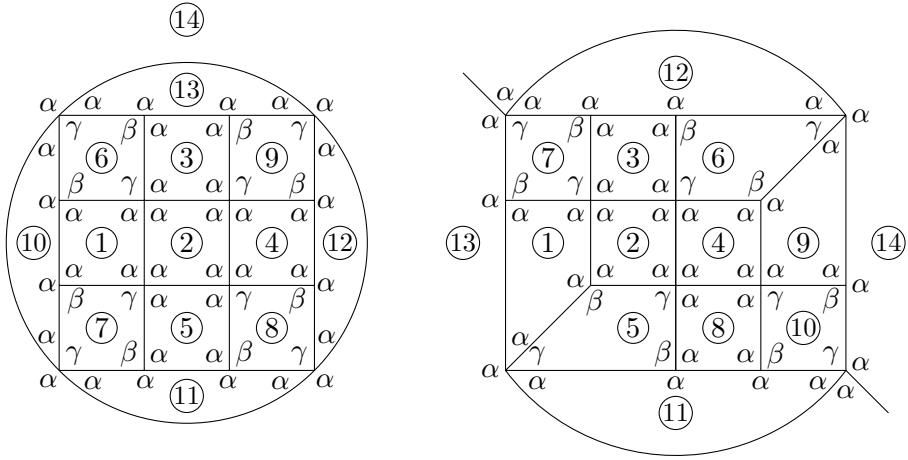
\begin{figure}[h!] 
\centering
\begin{tikzpicture}

\tikzmath{
\s=1;
\r=0.8;
\th=360/4;
\x=\r*cos(0.5*\th);
\R = sqrt(\x^2 + (3*\x)^2);
\aR = acos(3*\x/\R);
\RR = sqrt((2*\x)^2+(4*\x)^2);
\aRR = acos(4*\x/\RR);
\RRR=sqrt( (3*\x)^2 + (4*\x)^2 );
\aRRR=acos(3*\x/\RRR);
}

\begin{scope}[] 

\foreach \a in {0,...,3}{

\draw[rotate=\th*\a]
	(0.5*\th:\r) -- (1.5*\th:\r)
	(\x, \x) -- (\x, 3*\x)
	(-\x, \x) -- (-\x, 3*\x)
	(-\x, 3*\x) -- (\x, 3*\x)
	(-\x, 3*\x) -- (-3*\x, 3*\x) 
	(\x, 3*\x) -- (3*\x, 3*\x) 
	%
;

\node at (45+\th*\a:0.65*\r) {\small $\alpha$}; 

\node at (45+\th*\a:1.35*\r) {\small $\gamma$}; 
\node at (45+\th*\a:2.65*\r) {\small $\gamma$};
\node at (62.5+\th*\a:2.075*\r) {\small $\beta$};
\node at (27.5+\th*\a:2.075*\r) {\small $\beta$};

\node at (90-0.9*\aR+\th*\a: 1.08*\R) {\small $\alpha$}; 
\node at (0.625*\th+\th*\a: 2.8*\r) {\small $\alpha$};
\node at (90+0.9*\aR+\th*\a: 1.08*\R) {\small $\alpha$};
\node at (1.375*\th+\th*\a: 2.8*\r) {\small $\alpha$};

\node at (45+\th*\a:3.25*\r) {\small $\alpha$}; 

}

\foreach \aa in {0,...,3} { 

\tikzset{shift={(90+\th*\aa:2*\x)}}

\foreach \a in {0,...,3}{

\node at (45+\th*\a:0.65*\r) {\small $\alpha$};

}

}

\draw (0,0) circle (3*\r);

\node[inner sep=1,draw,shape=circle] at (0,2*\x) {\small $3$};
\node[inner sep=1,draw,shape=circle] at (-2*\x,0) {\small $1$};
\node[inner sep=1,draw,shape=circle] at (0,0) {\small $2$};
\node[inner sep=1,draw,shape=circle] at (2*\x,0) {\small $4$};
\node[inner sep=1,draw,shape=circle] at (0,-2*\x) {\small $5$};

\node[inner sep=1,draw,shape=circle] at (0.5*\th: 2*\r) {\small $9$};
\node[inner sep=1,draw,shape=circle] at (1.5*\th: 2*\r) {\small $6$};
\node[inner sep=1,draw,shape=circle] at (2.5*\th: 2*\r) {\small $7$};
\node[inner sep=1,draw,shape=circle] at (3.5*\th: 2*\r) {\small $8$};

\node[inner sep=0.25,draw,shape=circle] at (0,3.6*\x) {\footnotesize $13$};
\node[inner sep=0.25,draw,shape=circle] at (-3.6*\x,0) {\footnotesize $10$};
\node[inner sep=0.25,draw,shape=circle] at (0,-3.6*\x) {\footnotesize $11$};
\node[inner sep=0.25,draw,shape=circle] at (3.6*\x,0) {\footnotesize $12$};

\node[inner sep=0.25,draw,shape=circle] at (0,5.25*\x) {\footnotesize $14$};

\end{scope}

\begin{scope}[xshift=6.5*\s cm] 

\foreach \aa in {-1,1} {

\tikzset{shift={(\aa*\x,0)}}

\foreach \a in {0,...,3}{

\draw[rotate=\th*\a]
	(0.5*\th:\r) -- (1.5*\th:\r)
;

\node at (0.5*\th+\th*\a: 0.65*\r) {\small $\alpha$}; 

}

}

\foreach \aa in {0,1} {

\draw[rotate=2*\aa*\th]
	(0, \x) -- (0, 3*\x) 
	(0, 3*\x) -- (4*\x, 3*\x)
	(2*\x, \x) -- (4*\x, 3*\x)
	(2*\x, -\x) -- (4*\x, -\x)
	(4*\x, -\x) -- (4*\x, 3*\x)
	(0, -\x) -- (0, -3*\x) 
	(2*\x, -\x) -- (2*\x, -3*\x)
	(0, -3*\x) -- (2*\x, -3*\x)
	(2*\x, -3*\x) -- (4*\x, -3*\x)
	(4*\x, -\x) -- (4*\x, -3*\x)
	(4*\x, -3*\x) --  (5*\x, -4*\x)
	(4*\x, 3*\x) arc (90-\aRRR:90+\aRRR:\RRR)
;

\foreach \a in {0,...,3} {

\tikzset{rotate=2*\aa*\th, shift={(\x,-2*\x)}}

\node at (0.5*\th+\th*\a: 0.65*\r) {\small $\alpha$}; 
	
}

\foreach \a in {0,...,1} {

\tikzset{rotate=2*\aa*\th, shift={(3*\x,0)}}

\node at (2.5*\th+\th*\a: 0.65*\r) {\small $\alpha$}; 

}

\foreach \a in {0,1} {

\tikzset{rotate=2*\aa*\th, shift={(3*\x,-2*\x)}}

\node at (1.5*\th+2*\th*\a: 0.65*\r) {\small $\gamma$}; 
\node at (0.5*\th+2*\th*\a: 0.65*\r) {\small $\beta$}; 

}

\tikzset{rotate=2*\aa*\th} {

\node at (2.35*\x, 0.9*\x) {\small $\alpha$};
\node at (3.65*\x, 2.25*\x) {\small $\alpha$};

\node at (0.3*\x,2.6*\x) {\small $\beta$};
\node at (1.9*\x,1.4*\x) {\small $\beta$};
\node at (0.3*\x,1.35*\x) {\small $\gamma$};
\node at (3.25*\x,2.65*\x) {\small $\gamma$};

\node at (0, 3.3*\x) {\small $\alpha$};
\node at (-2*\x, 3.3*\x) {\small $\alpha$};
\node at (3.25*\x, 3.3*\x) {\small $\alpha$};
\node at (-3.35*\x, 3.3*\x) {\small $\alpha$};

\node at (4.35*\x, -\x) {\small $\alpha$};
\node at (4.35*\x, 3*\x) {\small $\alpha$};
\node at (-4*\x, 3.5*\x) {\small $\alpha$};
\node at (-4.35*\x, 2.85*\x) {\small $\alpha$};

}

}

\node[inner sep=1,draw,shape=circle] at (-3*\x,0) {\small $1$};
\node[inner sep=1,draw,shape=circle] at (-\x,0) {\small $2$};
\node[inner sep=1,draw,shape=circle] at (-\x,2*\x) {\small $3$};
\node[inner sep=1,draw,shape=circle] at (\x,0) {\small $4$};
\node[inner sep=1,draw,shape=circle] at (-\x,-2*\x) {\small $5$};
\node[inner sep=1,draw,shape=circle] at (\x,2*\x) {\small $6$};
\node[inner sep=1,draw,shape=circle] at (-3*\x,2*\x) {\small $7$};
\node[inner sep=1,draw,shape=circle] at (\x,-2*\x) {\small $8$};
\node[inner sep=1,draw,shape=circle] at (3*\x,0) {\small $9$};
\node[inner sep=0.25,draw,shape=circle] at (3*\x,-2*\x) {\footnotesize $10$};
\node[inner sep=0.25,draw,shape=circle] at (0,-4*\x) {\footnotesize $11$};
\node[inner sep=0.25,draw,shape=circle] at (0,4*\x) {\footnotesize $12$};
\node[inner sep=0.25,draw,shape=circle] at (-5*\x,0) {\footnotesize $13$};
\node[inner sep=0.25,draw,shape=circle] at (5*\x,0) {\footnotesize $14$};

\end{scope} 

\end{tikzpicture}
\caption{The two tilings with $ \alpha^2\beta, \alpha^3\gamma $ and $10$ squares and $4$ rhombi}
\label{Fig-a4-Tilings-al2be-al3ga}
\end{figure}

If it is a rhombus adjacent to $T_2$ in Figure \ref{Fig-a4-Tilings-al2be-al3ga} along the remaining edge, then we determine $T_5$ in the second picture. Apply the similar argument, we determine another tiling the same vertices and the same numbers of tiles in the second picture.
\end{proof}

\begin{prop}\label{Prop-be2ga} The dihedral tilings with vertex $\beta^2\gamma$ are an infinite family of tilings in Figure \ref{Fig-a4-Tilings-be2ga-albega2}.
\end{prop}

Each member of the infinite family has $2$ squares and $4(2c-1)$ rhombi, where $c\ge2$ is the number of $\gamma$'s in the vertex $\alpha\beta\gamma^c$. They are illustrated in Figure \ref{Fig-a4-families-earth-map-tilings}.

\begin{proof} By Proposition \ref{Prop-al2be}, we know that $\alpha^2\beta$ is not a vertex. By $\beta^2\gamma$ and $\beta>\alpha>\gamma$, we get $\beta^2\cdots=\beta^2\gamma$ and $\beta>\frac{2}{3}\pi$. Combined with $\alpha>\frac{1}{2}\pi$, we get the vertices below,
\begin{align*}
\beta^2\gamma, \gamma^c, \alpha^{a\le3}\gamma^c, \beta\gamma^c, \alpha\beta\gamma^c .
\end{align*}
From the above, we have $\alpha\beta\cdots = \alpha\beta\gamma^c$. Lemma \ref{Lem-albe-alga} implies that $\alpha\beta\cdots$ is a vertex. This means that $\alpha\beta\gamma^c$ is a vertex.

We first show that $\gamma^c, \alpha^a\gamma^c$ are not vertices for dihedral tilings. 

The arrangement of $\alpha\vert\gamma\cdots$ determines tiles $T_1, T_2$ in Figure \ref{Fig-a4-AAD-be2ga-alga}. Then $\alpha_1\beta_2\cdots=\alpha\beta\gamma^c$ determines $T_3$. Repeating the same argument we further determine $T_4,T_5$. Hence $\alpha\vert\gamma\cdots = \beta \vert \alpha \vert \gamma\cdots =\alpha\beta\gamma^c$. This further implies that $\alpha^a\gamma^c$ is not a vertex.

\begin{figure}[h!] 
\centering
\begin{tikzpicture}

\tikzmath{
\s=1;
\r=0.8;
\th=360/4;
\x=\r*cos(0.5*\th);
\R = sqrt(\x^2+(3*\x)^2);
\aR = acos(3*\x/\R);
}

\begin{scope}[] 

\foreach \aa in {0,...,3} {

\tikzset{shift={(\th+\aa*\th:2*\x)}}

\foreach \a in {0,...,3} {

\draw[rotate=\th*\a]
	(0.5*\th:\r) -- (1.5*\th:\r)
;

}

}

\foreach \a in {0,...,3} {

\node at (0.5*\th+\th*\a: 0.65*\r) {\small $\alpha$}; 

}

\foreach \aa in {1,3} {

\tikzset{shift={(\th+\aa*\th:2*\x)}}

\foreach \a in {0,2} {

\node at (0.5*\th+\th*\a: 0.65*\r) {\small $\gamma$}; 
\node at (1.5*\th+\th*\a: 0.65*\r) {\small $\beta$}; 

}

}

\foreach \aa in {0,2} {

\tikzset{shift={(\th+\aa*\th:2*\x)}}

\foreach \a in {0,2} {

\node at (0.5*\th+\th*\a: 0.65*\r) {\small $\beta$}; 
\node at (1.5*\th+\th*\a: 0.65*\r) {\small $\gamma$}; 

}

}

\node at (1.5*\x, 1.65*\x) {\small $\ddots$};
\node at (-1.5*\x, 1.65*\x) {\small $\iddots$};
\node at (-1.5*\x, -1.25*\x) {\small $\ddots$};
\node at (1.5*\x, -1.25*\x) {\small $\iddots$};

\node[inner sep=1,draw,shape=circle] at (0,0) {\small $1$};
\node[inner sep=1,draw,shape=circle] at (2*\x,0) {\small $2$};
\node[inner sep=1,draw,shape=circle] at (0,2*\x) {\small $3$};
\node[inner sep=1,draw,shape=circle] at (-2*\x,0) {\small $4$};
\node[inner sep=1,draw,shape=circle] at (0,-2*\x) {\small $5$};

\end{scope}

\end{tikzpicture}
\caption{The arrangement of $\alpha\vert\gamma$}
\label{Fig-a4-AAD-be2ga-alga}
\end{figure}

If $\gamma^c$ is a vertex, then its $\gamma\vert\gamma$ determines $T_1, T_2$ in Figure \ref{Fig-a4-be2-gac-Timezones}. By $\beta^2\cdots=\beta^2\gamma$, we also get $T_3$. This process continues at $\gamma^c$ and we determine a monohedral earth map tiling for each fixed $c\ge3$. Further details about the earth map tilings can be seen in \cite{cly}. Hence $\gamma^c$ is not a vertex for dihedral tilings.

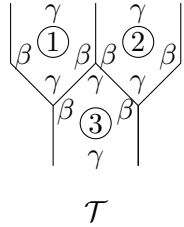
\begin{figure}[h!]
\centering
\begin{tikzpicture}[>=latex,scale=1]

\tikzmath{ 
\s=1;
\r=0.8;
\x = sqrt(2*\r^2);
\y = sqrt(\r^2 - (\x/2)^2);
\yy = 2*\r + \y;
\l=2*\r*sin(67.5);
\L=sqrt(\r^2+\l^2-2*\r*\l*cos(135+22.5));
\A=acos((\L^2+\l^2-\r^2)/(2*\L*\l));
}

\begin{scope}[] 

	(\x,0) -- (\x,\r) -- (90-22.5:\l) -- (90-22.5+\A:\L) -- (-0.5*\x, 2*\r+\y) -- (-0.5*\x, \r+\y) -- (0*\x, \r) -- (0,0) -- (\x, 0)
;

\tikzmath{ 
\tz=1;
\tzz=\tz-1;
\tzzz=\tzz-1;
}

\draw[]
	(90:\r) -- (90+22.5:\l)
	(90+22.5:\l) -- (90+22.5-\A:\L)
	(1*\x,0) -- (1*\x, \r)
;

\foreach \a in {0,...,\tz} {

\draw[xshift=\x*\a cm]
	(0:0) -- (90:\r)
	(90:\r) -- (90-22.5:\l)
	(90-22.5:\l) -- (90-22.5+\A:\L)
;
	
}

\foreach \a in {0,...,\tzz} {

\draw[xshift=\x*\a cm]
	(90-22.5:\l) -- (\x,\r)
;

}

\foreach \c in {-1,0,...,\tzz} {

\node at (1*\x+\c*\x, 2*\r + 0.75*\y) {\small $\gamma$}; 
\node at (1*\x+\c*\x, 1*\r + 0.5*\y) {\small $\gamma$}; 
\node at (0.65*\x+\c*\x, 1.1*\r + 1*\y) {\small $\beta$}; 
\node at (1.35*\x+\c*\x, 1.1*\r + 1*\y) {\small $\beta$}; 

}

\foreach \c in {0,...,\tzz} {

\node at (0.5*\x+\c*\x, 0.1*\r) {\small $\gamma$}; 
\node at (0.5*\x+\c*\x, 1*\r + 0.5*\y) {\small $\gamma$}; 
\node at (0.15*\x+\c*\x, 0.9*\r) {\small $\beta$}; 
\node at (0.85*\x+\c*\x, 0.9*\r) {\small $\beta$}; 

}

\node[inner sep=1,draw,shape=circle] at (0,\r+0.75*\x) {\small $1$};
\node[inner sep=1,draw,shape=circle] at (\x,\r+0.75*\x) {\small $2$};
\node[inner sep=1,draw,shape=circle] at (0.5*\x,0.5*\x) {\small $3$};

\node at (\y,-0.75*\r) {\small $\mathcal{T}$};

\end{scope} 

\end{tikzpicture}
\caption{A timezone given by $T_1, T_3$ and $\mathcal{T} = c-\frac{1}{2}$ timezones, $c=2$}
\label{Fig-a4-be2-gac-Timezones}
\end{figure}

For $c\ge2$, we use $\mathcal{T}$ to denote a block of $c-\frac{1}{2}$ timezones which consist of $c-1$ timezones and one extra tile. Examples with $c=2,3$ are illustrated in Figure \ref{Fig-a4-be2-gac-Timezones} and the first picture of Figure \ref{Fig-a4-families-earth-map-tilings} respectively.

By no $\gamma^c, \alpha^a\gamma^c$, we get
\begin{align}\label{Eq-AVC-be2ga-begac-albegac}
\AVC = \{ \beta^2\gamma, \beta\gamma^c, \alpha\beta\gamma^c \}.
\end{align} 
Since $\alpha$ appears at some vertex, we know that $\alpha\beta\gamma^c$ is a vertex. By Figure \ref{Fig-a4-AAD-be2ga-alga}, we get the central square and its four vertices in Figure \ref{Fig-a4-Tilings-be2ga-albega2}. For any fixed $c\ge2$, the $\gamma^c$-part of each $\alpha\beta\gamma^c$ determines $\mathcal{T}$ in Figure \ref{Fig-a4-Tilings-be2ga-albega2}. The tiling is then completed by another square in the exterior of the \quotes{circular boundary} in Figure \ref{Fig-a4-Tilings-be2ga-albega2}. The minimal member ($c=2$) of the family is given in the second picture.

\begin{figure}[h!] 
\centering
\begin{tikzpicture}

\tikzmath{
\s=1;
}

\begin{scope}[] 

\tikzmath{
\L=0.8;
\th=360/4;
\X=\L*cos(0.5*\th);
\l=2*\X;
}

\foreach \a in {0,...,3}
{
\begin{scope}[rotate=90*\a]

\draw
	(0.5*\l,-0.5*\l) -- (0.5*\l,0.5*\l) -- (1*\l,1*\l) -- (1*\l,2*\l) -- ++(-0.4*\l,-0.3*\l)
	(1*\l,1*\l) -- ++(0.3*\l,-0.3*\l);

\node at (0.3*\l,0.3*\l) {\small $\alpha$};
\node at (-0.4*\l,0.7*\l) {\small $\beta$};
\node at (0.35*\l,0.7*\l) {\small $\gamma^c$};
\node at (1*\l,0.7*\l) {\small $\gamma$};
\node at (1.2*\l,1.1*\l) {\small $\beta$};
\node at (0.8*\l,1.1*\l) {\small $\beta$};
\node at (0.8*\l,1.65*\l) {\small $\gamma$};
\node at (0.6*\l,1.95*\l) {\small $\beta$};
\node at (1.1*\l,2.15*\l) {\small $\alpha$};
\node at (0,1.4*\l) {\small $\mathcal{T}$};

\end{scope}
}	

\draw (0,0) circle (2.24*\l);

\node[rotate=-45] at (1.35*\l,1.5*\l) {\small $\gamma^{c-1}$};
\node[rotate=135] at (-1.35*\l,-1.45*\l) {\small $\gamma^{c-1}$};
\node[rotate=45] at (-1.5*\l,1.35*\l) {\small $\gamma^{c-1}$};
\node[rotate=-135] at (1.45*\l,-1.35*\l) {\small $\gamma^{c-1}$};

\end{scope}

\begin{scope}[xshift=6.5*\s cm] 

\tikzmath{
\r=0.85;
\th=360/4;
\x=\r*cos(0.5*\th);
\R = sqrt(\x^2 + (3*\x)^2);
\aR = acos(3*\x/\R);
}

\foreach \a in {0,...,3}{

;

\draw[rotate=\th*\a]
	(0.5*\th:\r) -- (1.5*\th:\r)
	(\x, \x) -- (\x, 3*\x)
	(-\x, \x) -- (-\x, 3*\x)
	(-\x, 3*\x) -- (\x, 3*\x)
	(-\x, 3*\x) -- (-3*\x, 3*\x) 
	(\x, 3*\x) -- (3*\x, 3*\x) 
	%
;

\node at (45+\th*\a:0.65*\r) {\small $\alpha$}; 

\node at (25+\th*\a:1.05*\r) {\small $\beta$}; 
\node at (65+\th*\a:1.05*\r) {\small $\gamma$};
\node at (76+\th*\a:1.9*\r) {\small $\beta$};
\node at (15+\th*\a:1.9*\r) {\small $\gamma$};

\node at (45+\th*\a:1.35*\r) {\small $\gamma$}; 
\node at (45+\th*\a:2.65*\r) {\small $\gamma$};
\node at (62.5+\th*\a:2.075*\r) {\small $\beta$};
\node at (27.5+\th*\a:2.075*\r) {\small $\beta$};


\node at (45+\th*\a:3.25*\r) {\small $\alpha$}; 

}

\draw (0,0) circle (3*\r);

\node at (0.65*\th: 2.725*\r) {\small $\beta$};
\node at (90-0.9*\aR: 1.1*\R) {\small $\gamma$};
\node at (90+0.9*\aR: 1.1*\R) {\small $\beta$};
\node at (1.375*\th: 2.8*\r) {\small $\gamma$};

\node at (0.6*\th+\th: 2.8*\r) {\small $\beta$};
\node at (90-0.9*\aR+\th: 1.075*\R) {\small $\gamma$};
\node at (90+0.9*\aR+\th: 1.075*\R) {\small $\beta$};
\node at (1.4*\th+\th: 2.8*\r) {\small $\gamma$};

\node at (0.65*\th+2*\th: 2.8*\r) {\small $\beta$};
\node at (90-0.9*\aR+2*\th: 1.1*\R) {\small $\gamma$};
\node at (90+0.9*\aR+2*\th: 1.125*\R) {\small $\beta$};
\node at (1.375*\th+2*\th: 2.8*\r) {\small $\gamma$};

\node at (0.6*\th+3*\th: 2.8*\r) {\small $\beta$};
\node at (90-0.9*\aR+3*\th: 1.075*\R) {\small $\gamma$};
\node at (90+0.9*\aR+3*\th: 1.075*\R) {\small $\beta$};
\node at (1.4*\th+3*\th: 2.8*\r) {\small $\gamma$};

\end{scope} 

\end{tikzpicture}
\caption{The tilings with $\beta^2\gamma, \alpha\beta\gamma^c$ and $2$ squares and $4(2c-1)$ rhombi}
\label{Fig-a4-Tilings-be2ga-albega2}
\end{figure}
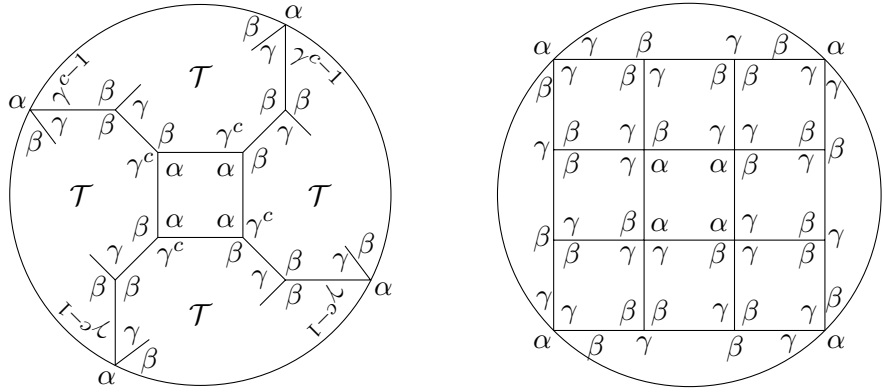


The infinite family of tilings have vertices below
\begin{align}\label{Eq-AVC-be2ga-albegac}
\AVC\equiv\{ \beta^2\gamma, \alpha\beta\gamma^c \}, \quad c\ge2.
\end{align}
Note that $\beta\gamma^c$ is not a vertex for dihedral tilings. There is a monohedral tiling the vertices $\beta^2\gamma, \beta\gamma^{c}$ in \cite{cly}.

Lastly, we show the geometric existence of these tilings. By $\beta^2\gamma, \alpha\beta\gamma^c$, we have
\begin{align} \label{Eq-be2ga-be-al}
\alpha = \pi - (c-\tfrac{1}{2})\gamma, \quad
\beta= \pi - \tfrac{1}{2}\gamma.
\end{align}
The geometric existence of the tilings means $\pi > \beta > \alpha > \frac{1}{2}\pi > \gamma$ and \eqref{Eq-tan-al-be-ga} is satisfied and $0 < \cos x < 1$.

By $\gamma > 0$ and $c\ge2$ and \eqref{Eq-be2ga-be-al}, we have $\beta>\alpha$. Moreover, $\alpha > \gamma$ is equivalent to
\begin{align} \label{Eq-be2ga-c-ga-ineq}
0<(c+\tfrac{1}{2})\gamma<\pi.
\end{align}
It implies $\gamma \in (0, \frac{1}{2}\pi)$. Substituting \eqref{Eq-be2ga-be-al} into \eqref{Eq-tan-al-be-ga}, we get
\begin{align}\label{Eq-be2ga-tan-c-ga}
2\tan^2 ( 2c-1 ) \tfrac{1}{4}\gamma + \tan^2 \tfrac{1}{4}\gamma  - 1 = 0.
\end{align}
From the proof of \eqref{Eq-tan-al-be-ga} (Lemma \ref{Lem-al-be-ga}), we know $\cos x = \cot \frac{1}{2}\beta \cot \frac{1}{2}\gamma$. Then for $\gamma \in (0, \frac{1}{2}\pi)$, the second equation of \eqref{Eq-be2ga-be-al} and \eqref{Eq-tan-al-be-ga} imply $0 < \cos x = \frac{1}{2}( 1 - \tan^2 \frac{\gamma}{4}) < 1$.

Then geometric existence of the tilings is equivalent to the existence of $\gamma$ for each integer $c\ge2$ such that both \eqref{Eq-be2ga-c-ga-ineq}, \eqref{Eq-be2ga-tan-c-ga} are satisfied and $\alpha>\frac{1}{2}\pi$.

The general solution to \eqref{Eq-be2ga-tan-c-ga} is 
\begin{align*}
c_{\pm}=\pm\tfrac{2}{\gamma} (T(\gamma) + k\pi )+\tfrac{1}{2}, \quad
\text{ where } T(\gamma) = \tan^{-1}\sqrt{\tfrac{\tan \frac{1}{4}\gamma}{\tan \frac{1}{2}\gamma}}, \ 
k\in\mathbb{Z}.
\end{align*}
For $\gamma \in (0, \frac{1}{2}\pi)$, we have $0 < T(\gamma) < \frac{1}{2}\pi$. This implies 
\begin{align*}
c_{+} &\in (\tfrac{1}{2}, \tfrac{\pi}{\gamma} + \tfrac{1}{2}) + \tfrac{2k\pi}{\gamma}, \\
c_{-} &\in (\tfrac{1}{2} - \tfrac{\pi}{\gamma}, \tfrac{1}{2}) - \tfrac{2k\pi}{\gamma}.
\end{align*}
Then by \eqref{Eq-be2ga-c-ga-ineq} we get $k=0$. For $k=0$, we have $c_{-} < \frac{1}{2}$, contradicting $c\ge2$. It remains to discuss $c_{+}$ for $k=0$, which gives
\begin{align} \label{Eq-be2ga-albegac-c(ga)}
c(\gamma) = \tfrac{2}{\gamma} T(\gamma) +\tfrac{1}{2}.
\end{align}

By the first equality in \eqref{Eq-be2ga-be-al}, the inequality $\alpha>\frac{1}{2}\pi$ is equivalent to $(c-\frac{1}{2})\gamma<\frac{1}{2}\pi$. By \eqref{Eq-be2ga-albegac-c(ga)}, it is equivalent to $T(\gamma)<\frac{1}{4}\pi$. Since $\tan t$ is strictly increasing on $(0,\frac{1}{2}\pi)$, for $\gamma\in (0, \frac{1}{2}\pi)$ we have $\tan\frac{1}{4}\gamma<\tan\frac{1}{2}\gamma$ and hence $T(\gamma) < \frac{1}{4}\pi$ holds.

It is straightforward to derive that $c(\gamma)$ in \eqref{Eq-be2ga-albegac-c(ga)} satisfies \eqref{Eq-be2ga-c-ga-ineq} if and only if $\tan \frac{1}{4}\gamma \tan \frac{1}{2}\gamma < 1$, which holds for $\gamma \in (0, \frac{1}{2}\pi)$.

We plot $c(\gamma)$ against $\frac{\gamma}{\pi}$ in Figure \ref{Fig-a4-be2ga-albegac-c(ga)}.
\begin{figure}[htp]
\centering
\includegraphics[scale=0.5]{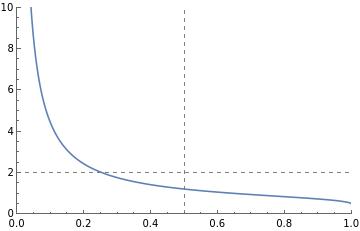}
\caption{Graph of $c(\gamma)$ against $\tfrac{\gamma}{\pi}$}
\label{Fig-a4-be2ga-albegac-c(ga)}
\end{figure}

It remains to show that, for any integer $c\ge2$, there is a $\gamma_c \in (0, \frac{1}{2}\pi)$ satisfying $c(\gamma_c)=c$. One can show that \eqref{Eq-be2ga-albegac-c(ga)} is continuous and decreasing on $(0, \frac{1}{2}\pi)$ and 
\begin{align*}
\lim_{\gamma \to 0^+} c(\gamma)= + \infty.
\end{align*}
On the other hand, we have $c(\frac{1}{2}\pi)\approx 1.228 < 2$. Then by Intermediate Value Theorem, the desired $\gamma_c$ exists and is unique.

Therefore the infinite family of tilings in Figure \ref{Fig-a4-Tilings-be2ga-albega2} exist.
\end{proof}

\end{document}